\documentclass[letterpaper,colorlinks=true,linkcolor=Blue2,citecolor=Green4,urlcolor=red]{scrartcl}
\usepackage[T1]{fontenc}
\usepackage{lmodern}
\usepackage{amsthm}
\usepackage{amsmath}
\usepackage{amssymb}
\usepackage{amsfonts}
\usepackage{mathtools}
\usepackage{mathrsfs}
\usepackage[x11names]{xcolor}
\usepackage{dsfont}
\usepackage{enumitem}
\usepackage{float}
\usepackage[totalwidth=480pt, totalheight=680pt]{geometry}
\usepackage{hyperref}

\theoremstyle{definition}
\newtheorem{definition}{Definition}[section]
\newtheorem{example}[definition]{Example}

\theoremstyle{plain}
\newtheorem{theorem}[definition]{Theorem}
\newtheorem{proposition}[definition]{Proposition}
\newtheorem{lemma}[definition]{Lemma}
\newtheorem{corollary}[definition]{Corollary}

\theoremstyle{remark}

\addtokomafont{disposition}{\normalfont\scshape}

\makeatletter
\xpatchcmd{\@sec@pppage}{
\bfseries}{
\normalfont\scshape\Large}{}{}
\makeatother

\setkomafont{section}{\Large}
\setkomafont{subsection}{\large}

\numberwithin{equation}{section}

\newcommand{\Cq}{C_{\textnormal{q}}(\mathds{R}^{d})}
\newcommand{\Cqaff}{C_{\textnormal{q}}^{\textnormal{aff}}(\mathds{R}^{d})}
\newcommand{\Cqaffo}{C_{\textnormal{q}}^{\textnormal{aff}}(\mathds{R})}
\newcommand{\PP}{\mathscr{P}}
\newcommand{\Rd}{\mathds{R}^{d}}
\newcommand{\lc}{\preceq_{\textnormal{c}}}
\newcommand{\R}{\mathds{R}}
\newcommand{\MCov}{\textnormal{MCov}}
\newcommand{\MT}{\mathsf{MT}}
\newcommand{\Cpl}{\mathsf{Cpl}}
\newcommand{\bary}{\operatorname{bary}}
\newcommand{\E}{\mathds{E}}
\newcommand{\Law}{\operatorname{Law}}
\newcommand{\dom}{\operatorname{dom}}

\newcommand{\sbm}{\textnormal{SBM}}
\newcommand{\supp}{\operatorname{supp}}
\newcommand{\dmt}{\textnormal{DMT}}
\newcommand{\psilim}{\psi_{\textnormal{lim}}}
\newcommand{\Aff}{\textnormal{Aff}}
\newcommand{\aff}{\textnormal{aff}}

\newcommand{\psiopt}{\psi_{\textnormal{opt}}}

\begin{document}


\title{
\LARGE The decomposition of stretched Brownian motion into Bass martingales\thanks{
This research was funded by the Austrian Science Fund (FWF) [Grant DOIs: 10.55776/P35197 and 10.55776/P35519]. For open access purposes, the authors have applied a CC BY public copyright license to any author accepted manuscript version arising from this submission.

\smallskip

\noindent We thank Pietro Siorpaes for his valuable feedback and comments during the preparation of this paper. 
}}


\author{
\large Walter Schachermayer\thanks{
Faculty of Mathematics, University of Vienna (email: \href{mailto: walter.schachermayer@univie.ac.at}{walter.schachermayer@univie.ac.at}).
} 
\and
\large Bertram Tschiderer\thanks{
Faculty of Mathematics, University of Vienna (email: \href{mailto: bertram.tschiderer@univie.ac.at}{bertram.tschiderer@univie.ac.at}).
} 
}


\date{}


\maketitle


\begin{abstract} \small \noindent \textsc{Abstract.} In previous work J.\ Backhoff-Veraguas, M.\ Beiglb\"ock and the present authors showed that the notions of stretched Brownian motion and Bass martingale between two probability measures on Euclidean space coincide if and only if these two measures satisfy an irreducibility condition.

Now we consider the general case, i.e.\ a pair of measures which are not necessarily irreducible. We show that there is a paving of Euclidean space into relatively open convex sets such that the stretched Brownian motion decomposes into a (possibly uncountable) family of Bass martingales on these sets. This paving coincides with the irreducible convex paving studied in previous work of H.\ De March and N.\ Touzi.

\bigskip

\small \noindent \href{https://mathscinet.ams.org/mathscinet/msc/msc2020.html}{\textit{MSC 2020 subject classifications:}} Primary 60G42, 60G44; secondary 91G20.

\bigskip

\small \noindent \textit{Keywords and phrases:} optimal transport, Brenier's theorem, Benamou--Brenier, stretched Brownian motion, Bass martingale
\end{abstract}

\section{Introduction} \label{sec_int_gen_case}

As in the parallel paper \cite{BVBST23} we consider two probability measures $\mu$ and $\nu$ on $\Rd$ with finite second moments, denoted by $\mu,\nu \in \PP_{2}(\Rd)$. It is always assumed that $\nu$ dominates $\mu$ in convex order, written $\mu \lc \nu$, and meaning that $\int f \, d\mu \leqslant \int f \, d\nu$ holds for all convex functions $f \colon \Rd \rightarrow \R$.

\smallskip

Our object of interest are \textit{martingale transports} starting from $\mu$ and terminating in $\nu$, i.e., one-step martingales $M = (M_{0},M_{1})$ with initial distribution $\Law(M_{0}) = \mu$ and terminal distribution $\Law(M_{1}) = \nu$. A celebrated result by Strassen \cite{Str65} shows that, under the above assumptions, there always exists a martingale transport from $\mu$ to $\nu$. We may and shall identify the one-step martingale $M = (M_{0},M_{1})$ with the joint law $\pi \in \PP_{2}(\Rd \times \Rd)$ of the random vector $(M_{0},M_{1})$. The set of all martingale transports $\pi$ from $\mu$ to $\nu$ is denoted by $\MT(\mu,\nu)$.

\smallskip

In \cite{BVBHK20} and \cite{BVBST23} a special martingale transport $\pi^{\sbm} \in \MT(\mu,\nu)$ has been identified: the \textit{stretched Brownian motion} from $\mu$ to $\nu$, which uniquely exists for all $\mu, \nu \in \PP_{2}(\Rd)$ with $\mu \lc \nu$. It is defined as the optimizer $\pi^{\sbm}(dx,dy) = \pi_{x}^{\sbm}(dy) \, \mu(dx)$ of the martingale optimal transport problem
\begin{equation} \label{pp_mbb_mcov_dt_f}
P(\mu,\nu) = \sup_{\pi \in \MT(\mu,\nu)}  \int \MCov(\pi_{x},\gamma) \, \mu(dx),
\end{equation}
where $\gamma$ is the standard Gaussian measure on $\Rd$. The maximal covariance $\MCov(\, \cdot \, , \, \cdot \,)$ between two probability measures $p,q \in \PP_{2}(\Rd)$ is defined as 
\[
\MCov(p,q) = \sup_{\tilde{\pi} \in \Cpl(p,q)} \int \langle y,z \rangle \, \tilde{\pi}(dy,dz), 
\]
where the set $\Cpl(p,q)$ consists of all couplings $\tilde{\pi}$ between $p$ and $q$, i.e., probability measures $\tilde{\pi}$ on $\Rd \times \Rd$ with marginals $p$ and $q$. By \cite[Theorem 2.2]{BVBHK20}, the discrete-time formulation \eqref{pp_mbb_mcov_dt_f} is equivalent to the continuous-time martingale optimal transport problem
\begin{equation} \label{pp_mbb_trace_ct_f}
P(\mu,\nu) =
\sup_{\substack{M_{0} \sim \mu, \, M_{1} \sim \nu, \\ M_{t} = M_{0} + \int_{0}^{t} \sigma_{s} \, dB_{s}}}
\mathds{E}\Big[\int_{0}^{1} \textnormal{tr} (\sigma_{t}) \, dt\Big],
\end{equation}
and there exists a unique-in-law optimizer $M^{\sbm} = (M^{\sbm}_{t})_{0 \leqslant t \leqslant 1}$ of \eqref{pp_mbb_trace_ct_f}. Furthermore, the optimizers $M^{\sbm}$ of \eqref{pp_mbb_trace_ct_f} and $\pi^{\sbm}$ of \eqref{pp_mbb_mcov_dt_f} are related via $\pi^{\sbm} = \Law(M^{\sbm}_{0},M^{\sbm}_{1})$, the values $P(\mu,\nu)$ in \eqref{pp_mbb_mcov_dt_f} and \eqref{pp_mbb_trace_ct_f} are equal, and $M^{\sbm}$ can be explicitly constructed from $\pi^{\sbm}$ (see Theorem 2.2 in \cite{BVBHK20}). Therefore the stretched Brownian motion from $\mu$ to $\nu$ can equivalently be viewed as the law $\pi^{\sbm} \in \MT(\mu,\nu)$ of the one step-martingale $(M^{\sbm}_{0},M^{\sbm}_{1})$, or as the continuous-time martingale $(M^{\sbm}_{t})_{0 \leqslant t \leqslant 1}$ with stochastic differential
\[
dM_{t}^{\sbm} = \sigma_{t}^{\sbm} \, dB_{t}, \qquad 0 \leqslant t \leqslant 1,
\]
where $(B_{t})_{0 \leqslant t \leqslant 1}$ is a standard Brownian motion on $\Rd$. In the following we shall freely switch between these two representations of a stretched Brownian motion. 

\smallskip

A particularly appealing subclass of stretched Brownian motions consists of \textit{Bass martingales}.

\begin{definition}[{\cite[Definition 1.1]{BVBST23}}] \label{def_bm_gcp} Let $(B_{t})_{0 \leqslant t \leqslant 1}$ be a $d$-dimensional Brownian motion with initial distribution $B_{0} \sim \alpha$, where $\alpha$ is a probability measure on $\Rd$. Let $v \colon \Rd \rightarrow \R$ be a convex function such that $\nabla v(B_{1})$ is square-integrable. The martingale 
\begin{equation} \label{eq_def_bm_gcp}
M_{t} \coloneqq 
\E[\nabla v(B_{1}) \, \vert \, \sigma(B_{s} \colon s \leqslant t)]
= \E[\nabla v(B_{1}) \, \vert \, B_{t}], \qquad 0 \leqslant t \leqslant 1,
\end{equation}
is called \textit{Bass martingale} with \textit{Bass measure} $\alpha \in \mathscr{P}(\Rd)$, initial law $\mu \coloneqq \Law(M_{0})$, and terminal law $\nu \coloneqq \Law(M_{1})$.
\end{definition}

Martingales of this form where introduced by Bass \cite{Bas83} (in dimension $d=1$ and with $\alpha$ a Dirac measure) in order to derive a solution of the Skorokhod embedding problem. Definition \ref{def_bm_gcp} generalizes this concept to multiple dimensions and non-degenerate starting laws. A Bass martingale is a stretched Brownian motion \cite[Theorem 1.10]{BVBHK20}, but not vice versa.

\subsection{The irreducible case}

In \cite{BVBST23} the following question was answered: under which conditions on $\mu$ and $\nu$ is the stretched Brownian motion $\pi^{\sbm}$ from $\mu$ to $\nu$ a Bass martingale? The answer is given by a connectivity assumption on the marginals, called \textit{irreducibility}.

\begin{definition}[{\cite[Definition 1.2]{BVBST23}}] For probability measures $\mu, \nu$ on $\Rd$ we say that the pair $(\mu,\nu)$ is \textit{irreducible}, if for all measurable sets $A, B \subseteq \Rd$ with $\mu(A) > 0$ and $\nu(B) > 0$, there exists a martingale transport $\pi \in \MT(\mu, \nu)$ such that $\pi(A \times B )>0$. 
\end{definition} 

The irreducibility assumption is necessary and sufficient for a stretched Brownian motion to be a Bass martingale, as shown in \cite[Theorem 1.3]{BVBST23}. The analysis of stretched Brownian motion and Bass martingales relies on the following dual formulation of the optimization problem \eqref{pp_mbb_mcov_dt_f}.

\begin{theorem}[{\cite[Theorem 1.4]{BVBST23}}] \label{theorem.1.4.o1a} Assume that $\mu, \nu \in \PP_{2}(\Rd)$ are in convex order. The value $P(\mu,\nu)$ of the problem \eqref{pp_mbb_mcov_dt_f} is equal to 
\begin{equation} \label{dp_mbb_mcov_dt_f}
D(\mu,\nu) = 
\inf_{\substack{\psi \in L^{1}(\nu), \\ \textnormal{$\psi$ convex}}} 
\Big( \int \psi \, d\nu - \int (\psi^{\ast} \ast \gamma)^{\ast} \, d \mu \Big).
\end{equation}
The infimum is attained by a lower semicontinuous convex function $\psi \colon \Rd \rightarrow (-\infty,+\infty]$ with $\mu(\operatorname{ri}(\dom\psi)) = 1$ if and only if $(\mu,\nu)$ is irreducible. In this case the (unique) optimizer to \eqref{pp_mbb_trace_ct_f} is given by the Bass martingale
\[
M_{t} \coloneqq \mathds{E}[\nabla v(B_{1}) \, \vert \, \sigma(B_{s} \colon s \leqslant t)]
= \mathds{E}[\nabla v(B_{1}) \, \vert \, B_{t}], \qquad 0 \leqslant t \leqslant 1,
\]
where $v = \psi^{\ast}$ and $B_{0} \sim \nabla (\psi^{\ast} \ast \gamma)^{\ast}(\mu)$.
\end{theorem} 

Here, the symbol $\ast$ used as a superscript denotes the convex conjugate of a function, otherwise it is the convolution operator; $\operatorname{ri}$ stands for relative interior, and $\dom$ for the domain of a function.

\smallskip

A technical caveat seems in order: it turns out that the optimizer $\psi$ of \eqref{dp_mbb_mcov_dt_f} is not necessarily $\nu$-integrable. Therefore, in order for the difference of the integrals in \eqref{dp_mbb_mcov_dt_f} to be well-defined, attainment of $D(\mu,\nu)$ has to be understood in a ``relaxed'' sense frequently encountered in martingale optimal transport problems; see \cite{BJ16, BNT17, BNS22} and Subsection \ref{subsec_rfotdf} below. 

\smallskip

The following simple observation will be of crucial importance for the analysis of the dual problem \eqref{dp_mbb_mcov_dt_f}: the value of the dual function 
\begin{equation} \label{def_dual_func_gen}
\psi \longmapsto \mathcal{D}(\psi) \coloneqq \int \psi \, d\nu - \int (\psi^{\ast} \ast \gamma)^{\ast} \, d \mu
\end{equation}
remains unchanged if the convex function $\psi$ is replaced by $\psi + \textnormal{aff}$, where $\textnormal{aff} \colon \Rd \rightarrow \R$ is an arbitrary affine function. As a consequence, dual optimizers $\psi$ --- provided they exist --- turn out to be unique, modulo adding affine functions. We thus make the following definition.

\begin{definition} \label{def_gc_ecocfeo} Let $\psi \colon \Rd \rightarrow \mathds{R}$ be a convex function. We write $[\psi]$ for the equivalence class of nonnegative convex functions of the form $\psi + \textnormal{aff}$, where $\textnormal{aff} \colon \Rd \rightarrow \R$ is an affine function.
\end{definition}

We make one more observation which will turn out to be useful: for an optimizing sequence $(\psi_{n})_{n \geqslant 1}$ of lower semicontinuous convex functions for the dual problem \eqref{dp_mbb_mcov_dt_f} only their values on the closed convex hull $\widehat{\supp}(\nu)$ of the support of $\nu$ matter. Indeed, defining
\[
\check{\psi}_{n} \coloneqq 
\begin{cases}
\psi_{n}, & \text{ on } \widehat{\supp}(\nu), \\
+ \infty, & \text{ on } \Rd \setminus \widehat{\supp}(\nu),
\end{cases}
\]
we obtain lower semicontinuous convex functions. It follows from the definition \eqref{def_dual_func_gen} of the dual function that $\mathcal{D}(\check{\psi}_{n}) \leqslant \mathcal{D}(\psi_{n})$, so that $(\check{\psi}_{n})_{n \geqslant 1}$ is an optimizing sequence, too. 

\smallskip

The following theorem is an extension of \cite[Theorem 1.4]{BVBST23}. It characterizes the features of the irreducible case. 

\begin{theorem}[\textsc{Irreducible case}] \label{irr_theom_gen_case_rep} Let $\mu, \nu \in \PP_{2}(\Rd)$ with $\mu \lc \nu$. We denote by $C \coloneqq \widehat{\supp}(\nu)$ the closed convex hull of the support of $\nu$ and by $I \coloneqq \operatorname{ri} C$ its relative interior. Let $(\psi_{n})_{n \geqslant 1}$ be an optimizing sequence of convex functions for the dual problem \eqref{dp_mbb_mcov_dt_f}, satisfying the regularity assumption \eqref{def_eq_cqaff}.
\begin{enumerate}[label=(\arabic*)] 
\item \label{irr_theom_gen_case_rep_1} If the pair $(\mu,\nu)$ is irreducible, then the following assertions hold:
\begin{enumerate}[label=(\roman*)] 
\item \label{irr_theom_gen_case_rep_o} $\mu(I) = 1$ and $\widehat{\supp}(\nu) = C^{\textnormal{B}}(x)$, for $\mu$-a.e.\ $x \in \Rd$, where $C^{\textnormal{B}}(x) \coloneqq \overline{I^{\textnormal{B}}(x)}$ and the set $I^{\textnormal{B}}(x)$ is defined in \eqref{eq_fdotbb_int} below.
\item \label{irr_theom_gen_case_rep_i} There exist representatives $\bar{\psi}_{n} \in [\psi_{n}]$ and there is a lower semicontinuous convex function $\psilim \colon \Rd \rightarrow [0,+\infty]$, which is unique modulo adding affine functions, such that 
\begin{alignat}{2}
\lim_{n \rightarrow \infty} \bar{\psi}_{n}(y) &= \psilim(y) < + \infty, \qquad &&y \in I, \\
\lim_{n \rightarrow \infty} \bar{\psi}_{n}(y) &= \psilim(y) = + \infty, \qquad &&y \in \Rd \setminus C.
\end{alignat}
\item \label{irr_theom_gen_case_rep_ii} The function $\psilim$ is a dual optimizer for the pair $(\mu,\nu)$.
\item \label{irr_theom_gen_case_rep_iii} There is a Bass martingale $M = (M_{t})_{0 \leqslant t \leqslant 1}$ from $\mu$ to $\nu$ as defined in \eqref{eq_def_bm_gcp}, with 
\begin{equation} \label{bomeva}
v \coloneqq \psilim^{\ast} 
\qquad \textnormal{ and } \qquad 
\alpha \coloneqq \nabla (\psilim^{\ast} \ast \gamma)^{\ast}(\mu).
\end{equation}
\end{enumerate}
\item \label{irr_theom_gen_case_rep_2} Conversely, if $\mu(I) = 1$ and the sequence $(\psi_{n}(y))_{n \geqslant 1}$ is bounded, for all $y \in I$, then the pair $(\mu,\nu)$ is irreducible and again the assertions \ref{irr_theom_gen_case_rep_o} -- \ref{irr_theom_gen_case_rep_iii} hold. 
\end{enumerate}
\end{theorem}

Most parts of this theorem were already shown in \cite{BVBST23}. We will see in Section \ref{sec_1b_potmr} that Theorem \ref{irr_theom_gen_case_rep} can be viewed as a special case of Theorem \ref{gen_case_theom_gen_case_int} below, which describes the non-irreducible situation.

\subsection{The main theorem: general case}

We turn to the case of a pair $(\mu,\nu)$ which is not necessarily irreducible. We refer to Section \ref{section_examples} for several illustrative examples of this case. As indicated by the title of the present paper, our main result is that the stretched Brownian motion $\pi^{\sbm} \in \MT(\mu,\nu)$ decomposes into a family of Bass martingales. For this purpose, we fix an optimizing sequence $(\psi_{n})_{n \geqslant 1}$ of convex functions for the dual problem \eqref{dp_mbb_mcov_dt_f}, satisfying the regularity assumption \eqref{def_eq_cqaff}. We also have to ensure that the sequence $(\mathcal{D}(\psi_{n}))_{n \geqslant 1}$ decreases sufficiently fast to the optimal value $D(\mu,\nu)$, in the sense that
\begin{equation} \label{eq_gen_case_fcc_fti}
\sum_{n \geqslant 1} \big( \mathcal{D}(\psi_{n}) - D(\mu,\nu) \big) < + \infty.
\end{equation}
We derive from the sequence $(\psi_{n})_{n \geqslant 1}$ the \textit{Bass paving} $\{I^{\textnormal{B}}(x)\}_{x \in \Rd}$ of $\Rd$ into disjoint relatively open convex sets, which are irreducible under any $\pi \in \MT(\mu,\nu)$, via
\begin{equation} \label{eq_fdotbb_int}
I^{\textnormal{B}}(x) \coloneqq \operatorname{ri} 
\bigg\{ y \in \Rd \colon \ \sup_{n \geqslant 1} \, \sup_{\substack{\bar{\psi}_{n} \in [\psi_{n}], \\ \bar{\psi}_{n}(x) \leqslant 1}} \, \bar{\psi}_{n}(y) < + \infty \bigg\}, \qquad x \in \Rd.
\end{equation}
In Propositions \ref{prop_sbm_paving} and \ref{prop_psinbar_bound} below we will show that \eqref{eq_fdotbb_int} indeed defines a paving of $\Rd$ enjoying the following properties:
\begin{enumerate}[label=(\roman*)] 
\item $x \in I^{\textnormal{B}}(x)$,
\item $I^{\textnormal{B}}(x) = I^{\textnormal{B}}(x_{1})$ for $x_{1} \in I^{\textnormal{B}}(x)$,
\item $I^{\textnormal{B}}(x) \cap I^{\textnormal{B}}(x_{1}) = \varnothing$ for $x_{1} \notin I^{\textnormal{B}}(x)$.
\end{enumerate}

\smallskip

We can thus decompose the measure $\mu$ into its restrictions $\mu^{I^{\textnormal{B}}}$ to the sets $I^{\textnormal{B}}$. For example, if there are at most countably many elements $\{I_{j}^{\textnormal{B}}\}_{j \geqslant 1}$, we have
\[
\mu(\, \cdot \,) = \sum_{j=1}^{\infty} \mu(I_{j}^{\textnormal{B}}) \, \mu^{I_{j}^{\textnormal{B}}}(\, \cdot \,), \qquad 
\mu^{I_{j}^{\textnormal{B}}}(\, \cdot \,) \coloneqq \frac{1}{\mu(I_{j}^{\textnormal{B}})} \, \mu\vert_{I_{j}^{\textnormal{B}}}(\, \cdot \,).
\]

\smallskip

In general, the Bass paving $\{I^{\textnormal{B}}(x)\}_{x \in \Rd}$ consists of uncountably many elements, see Example \ref{example_circles} below. Following \cite{DMT19}, we can still decompose the probability measure $\mu$ as 
\begin{equation} \label{eq_gen_case_deco_mu_fit}
\mu(\, \cdot \,) 
= \int_{\widehat{\mathcal{K}}} \, \mu^{I^{\textnormal{B}}}(\, \cdot \,) \ \kappa^{\textnormal{B}}(dI^{\textnormal{B}}),
\end{equation}
where $\kappa^{\textnormal{B}}$ is a Borel probability measure on the Polish space $\widehat{\mathcal{K}}$ of closed convex subsets of $\Rd$, equipped with the Wijsman topology \cite{Bee91}. Here, the collection $\widehat{\mathcal{K}}$ of closed convex sets is isomorphically identified with the collection $\operatorname{ri} \widehat{\mathcal{K}} \coloneqq \{ \operatorname{ri} C \colon C \in \widehat{\mathcal{K}}\}$ of their relative interiors. For each $x \in \Rd$, we denote by $C^{\textnormal{B}}(x)$ the closure of $I^{\textnormal{B}}(x)$.

\smallskip

Using the stretched Brownian motion $\pi^{\sbm} \in \MT(\mu,\nu)$, we define
\begin{equation} \label{eq_gen_case_deco_nu_fit}
\nu^{I^{\textnormal{B}}}(\, \cdot \,)  \coloneqq 
\int_{I^{\textnormal{B}}} \, \pi_{x}^{\sbm}(\, \cdot \,)  \ \mu^{I^{\textnormal{B}}}(dx).
\end{equation}
We thus obtain a decomposition of the probability measure $\nu$ via
\begin{equation} \label{eq_gen_case_deco_nu_fit_decomp}
\nu(\, \cdot \,) 
= \int_{\widehat{\mathcal{K}}} \, \nu^{I^{\textnormal{B}}}(\, \cdot \,) \ \kappa^{\textnormal{B}}(dI^{\textnormal{B}}).
\end{equation}

\smallskip

We now can formulate our main result.

\begin{theorem}[\textsc{General case}] \label{gen_case_theom_gen_case_int} Let $\mu, \nu \in \PP_{2}(\Rd)$ with $\mu \lc \nu$. Let $(\psi_{n})_{n \geqslant 1}$ be a sequence of convex functions satisfying the regularity assumption \eqref{def_eq_cqaff} and the ``fast convergence condition'' \eqref{eq_gen_case_fcc_fti}. Then there is a Bass paving $\{I^{\textnormal{B}}(x)\}_{x \in \Rd}$, a Borel probability measure $\kappa^{\textnormal{B}}$ on $\widehat{\mathcal{K}}$, and $\kappa^{\textnormal{B}}$-a.s.\ unique families of probability measures $\{\mu^{I^{\textnormal{B}}}\}_{I^{\textnormal{B}} \in \, \widehat{\mathcal{K}}}$ and $\{\nu^{I^{\textnormal{B}}}\}_{I^{\textnormal{B}} \in \, \widehat{\mathcal{K}}}$ in $\PP_{2}(\Rd)$, as in \eqref{eq_gen_case_deco_mu_fit} and \eqref{eq_gen_case_deco_nu_fit} above, such that the following assertions hold for $\kappa^{\textnormal{B}}$-a.e.\ $I^{\textnormal{B}} \in \widehat{\mathcal{K}}$:
\begin{enumerate}[label=(\roman*)] 
\item \label{gen_theom_gen_case_rep_o} $\mu^{I^{\textnormal{B}}} \lc \nu^{I^{\textnormal{B}}}$, $\mu^{I^{\textnormal{B}}}(I^{\textnormal{B}}) = 1$, and $C^{\textnormal{B}} \coloneqq \overline{I}^{\textnormal{B}} = \widehat{\supp}(\nu^{I^{\textnormal{B}}})$.
\item \label{gen_theom_gen_case_rep_i} There exist representatives $\bar{\psi}_{n}^{I^{\textnormal{B}}} \in [\psi_{n}]$ and there is a lower semicontinuous convex function $\psilim^{I^{\textnormal{B}}} \colon \Rd \rightarrow [0,+\infty]$, which is unique modulo adding affine functions, such that 
\begin{alignat}{2}
\lim_{n \rightarrow \infty} \bar{\psi}_{n}^{I^{\textnormal{B}}}(y) &= \psilim^{I^{\textnormal{B}}}(y) < + \infty, \qquad &&y \in I^{\textnormal{B}}, \label{gen_case_theom_gen_case_int_conv_i} \\
\lim_{n \rightarrow \infty} \bar{\psi}_{n}^{I^{\textnormal{B}}}(y) &= \psilim^{I^{\textnormal{B}}}(y) = + \infty, \qquad &&y \in \Rd \setminus C^{\textnormal{B}}. \label{gen_case_theom_gen_case_int_conv_ii}
\end{alignat}
\item \label{gen_theom_gen_case_rep_ii} The function $\psilim^{I^{\textnormal{B}}}$ is a dual optimizer for the pair $(\mu^{I^{\textnormal{B}}},\nu^{I^{\textnormal{B}}})$.
\item \label{gen_theom_gen_case_rep_iii} There is a Bass martingale $M^{I^{\textnormal{B}}} = (M_{t}^{I^{\textnormal{B}}})_{0 \leqslant t \leqslant 1}$ from $\mu^{I^{\textnormal{B}}}$ to $\nu^{I^{\textnormal{B}}}$ as defined in \eqref{eq_def_bm_gcp}, with 
\begin{equation} \label{eq_gen_theom_gen_case_rep_iii}
v^{I^{\textnormal{B}}} \coloneqq (\psilim^{I^{\textnormal{B}}})^{\ast}
\qquad \textnormal{ and } \qquad 
\alpha^{I^{\textnormal{B}}} \coloneqq \nabla \big((\psilim^{I^{\textnormal{B}}})^{\ast} \ast \gamma\big)^{\ast}(\mu^{I^{\textnormal{B}}}).    
\end{equation}
The Bass martingale $M^{I^{\textnormal{B}}}$ coincides with the stretched Brownian motion $\pi^{\sbm}$ restricted to $(\mu^{I^{\textnormal{B}}},\nu^{I^{\textnormal{B}}})$.
\item \label{gen_theom_gen_case_rep_iv} The pair $(\mu^{I^{\textnormal{B}}},\nu^{I^{\textnormal{B}}})$ is irreducible.
\end{enumerate}
\end{theorem}

A remarkable message of Theorem \ref{gen_case_theom_gen_case_int} is that an \textit{arbitrary} sequence $(\psi_{n})_{n \geqslant 1}$ of convex functions of class \eqref{def_eq_cqaff}, such that $(\mathcal{D}(\psi_{n}))_{n \geqslant 1}$ converges sufficiently fast (see \eqref{eq_gen_case_fcc_fti}) to the optimal value $D(\mu,\nu)$, already contains the \textit{full information} about the structure \eqref{eq_gen_theom_gen_case_rep_iii} of the family of Bass martingales $M^{I^{\textnormal{B}}}$ from $\mu^{I^{\textnormal{B}}}$ to $\nu^{I^{\textnormal{B}}}$. Note that in \eqref{gen_case_theom_gen_case_int_conv_i} and \eqref{gen_case_theom_gen_case_int_conv_ii} we did \textit{not} have to pass to subsequences or convex combinations of the sequence $(\psi_{n})_{n \geqslant 1}$ in order to obtain pointwise convergence. 

\medskip

We can also show that the Bass paving $\{I^{\textnormal{B}}(x)\}_{x \in \Rd}$ satisfies the maximality condition \eqref{gen_case_cor_gen_case_int_dmt_eq} below, as defined in \cite[Theorem 2.1]{DMT19}. In particular, Corollary \ref{gen_case_cor_gen_case_int_dmt} states that the martingale transport $\pi^{\sbm} \in \MT(\mu,\nu)$ is an element of the class of maximal transports $\pi \in \MT(\mu,\nu)$, considered by H.\ De March and N.\ Touzi in \cite{DMT19} (compare also \cite{OS17}).

\begin{corollary} \label{gen_case_cor_gen_case_int_dmt} Let $\pi \in \MT(\mu,\nu)$ be an arbitrary martingale transport. For $\mu$-a.e.\ $x \in \Rd$, we have that
\begin{equation} \label{gen_case_cor_gen_case_int_dmt_eq}
\widehat{\supp}(\pi_{x}) \subseteq \widehat{\supp}(\pi_{x}^{\sbm}) = C^{\textnormal{B}}(x).
\end{equation}
\end{corollary}

\medskip

For the convenience of the reader we will present a heuristic outline of the proof of Theorem \ref{gen_case_theom_gen_case_int} in the next subsection. The rigorous proofs of Theorem \ref{gen_case_theom_gen_case_int} and Corollary \ref{gen_case_cor_gen_case_int_dmt} are provided in Section \ref{sec_1b_potmr}.

\subsection{A guided tour through the proof of Theorem \ref{gen_case_theom_gen_case_int}}

Let $(\psi_{n})_{n \geqslant 1}$ be a sequence of convex functions satisfying the regularity assumption \eqref{def_eq_cqaff} and the ``fast convergence condition'' \eqref{eq_gen_case_fcc_fti}. We define the Bass paving $\{I^{\textnormal{B}}(x)\}_{x \in \Rd}$ as in \eqref{eq_fdotbb_int} above. We will show in Propositions \ref{prop_sbm_paving} and \ref{prop_psinbar_bound} below that this indeed defines a paving of $\Rd$ into relatively open convex sets.

\smallskip

Next, we show that the function $C^{\textnormal{B}} \colon \Rd \rightarrow \widehat{\mathcal{K}}$ given by 
\begin{equation} \label{eq_gen_int_map_bm}
x \longmapsto C^{\textnormal{B}}(x) = \overline{I^{\textnormal{B}}(x)},
\end{equation}
where $\widehat{\mathcal{K}}$ is equipped with the Wijsman topology, is Borel measurable. We then define $\kappa^{\textnormal{B}} \coloneqq C^{\textnormal{B}}(\mu)$ to be the pushforward measure of $\mu$ under the function \eqref{eq_gen_int_map_bm}, which gives a Borel probability measure on $\widehat{\mathcal{K}}$. Applying the disintegration theorem yields the $\kappa^{\textnormal{B}}$-a.s.\ unique family of probability measures $\{\mu^{I^{\textnormal{B}}}\}_{I^{\textnormal{B}} \in \, \widehat{\mathcal{K}}}$ in $\PP_{2}(\Rd)$, as in \eqref{eq_gen_case_deco_mu_fit} above. In the following, all statements have to be understood for $\mu$-a.e.\ $x \in \Rd$, or for $\kappa^{\textnormal{B}}$-a.e.\ $I^{\textnormal{B}} \in \widehat{\mathcal{K}}$. 

\smallskip

For fixed $I^{\textnormal{B}} \in \widehat{\mathcal{K}}$, we use $\mu^{I^{\textnormal{B}}}$ and the stretched Brownian motion $\pi^{\sbm} \in \MT(\mu,\nu)$ to define $\nu^{I^{\textnormal{B}}}$ as in \eqref{eq_gen_case_deco_nu_fit} above. Restricting the stretched Brownian motion $\pi^{\sbm}$ to $(\mu^{I^{\textnormal{B}}},\nu^{I^{\textnormal{B}}})$ yields the martingale transport
\begin{equation} \label{eq_def_pisbm_loc}
\pi^{I^{\textnormal{B}}}(dx,dy) 
\coloneqq \pi^{\sbm}_{x}(dy) \, \mu^{I^{\textnormal{B}}}(dx) \in \MT(\mu^{I^{\textnormal{B}}},\nu^{I^{\textnormal{B}}})
\end{equation}
from $\mu^{I^{\textnormal{B}}}$ to $\nu^{I^{\textnormal{B}}}$, so that 
\begin{equation} \label{eq_def_pisbm_loc_i}
\pi^{\sbm}(\, \cdot \,)  
= \int_{\widehat{\mathcal{K}}} \, \pi^{I^{\textnormal{B}}}(\, \cdot \,) \ \kappa^{\textnormal{B}}(dI^{\textnormal{B}}).
\end{equation}
It is not hard to verify (see Lemma \ref{lem_dotpsbmp} below) that $\pi^{I^{\textnormal{B}}}$ actually is the stretched Brownian motion from $\mu^{I^{\textnormal{B}}}$ to $\nu^{I^{\textnormal{B}}}$. The crucial result is that it is, in fact, a Bass martingale from $\mu^{I^{\textnormal{B}}}$ to $\nu^{I^{\textnormal{B}}}$. Let us sketch the argument how to find the dual optimizers corresponding to this family of Bass martingales.

\smallskip

For $x \in \Rd$, we consider representatives $\psi_{n}^{x} \in [\psi_{n}]$ such that the sequence $(\psi_{n}^{x}(x))_{n \geqslant 1}$ is bounded. For example, by Hahn--Banach, we can choose $\psi_{n}^{x} \in [\psi_{n}]$ so that $\psi_{n}^{x}(x) = 0$. By definition \eqref{eq_fdotbb_int}, for $x \in \Rd$, the representatives $(\psi_{n}^{x})_{n \geqslant 1}$ are pointwise bounded on $I^{\textnormal{B}}(x)$. We can even obtain pointwise \textit{convergence} on $I^{\textnormal{B}}(x)$ to a limiting function $\psilim^{I^{\textnormal{B}}(x)} \colon I^{\textnormal{B}}(x) \rightarrow [0,+\infty)$ by choosing (possibly different) representatives $\bar{\psi}_{n}^{x} \in [\psi_{n}]$. A possible way to find these representatives $\bar{\psi}_{n}^{x}$ is to choose a maximal set $\{ \xi_{1}, \ldots, \xi_{m}\}$ of affinely independent points in $I^{\textnormal{B}}(x)$ and prescribe arbitrary real values $\{a_{1}, \ldots, a_{m}\}$ at these points. Then there are affine functions $\textnormal{aff}_{n} \colon \Rd \rightarrow \R$ such that $\bar{\psi}_{n}^{x} \coloneqq \psi_{n}^{x} + \textnormal{aff}_{n}$ attains these values on $\xi_{1}, \ldots, \xi_{m}$, i.e.,
\begin{equation} \label{eq_aff_con_int}
\bar{\psi}_{n}^{x}(\xi_{j}) = a_{j}, \qquad j = 1, \ldots, m.
\end{equation}
The functions $\textnormal{aff}_{n}$ are uniquely determined by \eqref{eq_aff_con_int} on the affine span of $I^{\textnormal{B}}(x)$. By properly choosing $\{a_{1}, \ldots, a_{m}\}$, we can also make sure that $\bar{\psi}_{n}^{x} \geqslant 0$, so that $\bar{\psi}_{n}^{x} \in [\psi_{n}]$. Relying on \cite[Proposition 7.20]{BVBST23}, we come to a crucial point: it turns out that this sequence of representatives $(\bar{\psi}_{n}^{x})_{n \geqslant 1}$ indeed \textit{converges} pointwise to some $\psilim^{I^{\textnormal{B}}(x)}$ as in \eqref{gen_case_theom_gen_case_int_conv_i} and \eqref{gen_case_theom_gen_case_int_conv_ii}.

\smallskip

Finally, we show that the limiting function $\psilim^{I^{\textnormal{B}}} \coloneqq \psilim^{I^{\textnormal{B}}(x)}$ is a dual optimizer for the pair $(\mu^{I^{\textnormal{B}}},\nu^{I^{\textnormal{B}}})$. Theorem \ref{theorem.1.4.o1a} above then implies that there is a Bass martingale $(M_{t}^{I^{\textnormal{B}}})_{0 \leqslant t \leqslant 1}$ from $\mu^{I^{\textnormal{B}}}$ to $\nu^{I^{\textnormal{B}}}$, given by \eqref{eq_gen_theom_gen_case_rep_iii}. The resulting family of Bass martingales then decomposes the stretched Brownian motion $\pi^{\sbm} \in \MT(\mu,\nu)$ via \eqref{eq_def_pisbm_loc}, \eqref{eq_def_pisbm_loc_i}.

\subsection{Literature}

Pavings $\{ I(x) \}_{x \in \Rd}$ associated to martingale transports $\pi \in \MT(\mu,\nu)$ into relatively open convex sets, which are invariant under $\pi$ --- such as the above Bass paving $\{ I^{\textnormal{B}}(x) \}_{x \in \Rd}$ --- have been studied, notably in \cite{OS17, BJ16, DMT19, GKL19, Cio23a, Cio23b}. 

\smallskip 

In \cite{GKL19}, N.\ Ghoussoub, Y.-H.\ Kim and T.\ Lim focused on such decompositions, for a fixed martingale transport $\pi \in \MT(\mu,\nu)$. They obtained descriptions for the minimizers and maximizers for the cost function $c(x,y)=\vert x-y \vert$, when marginals are supported on $\R^2$, as well as for marginals on higher-dimensional state spaces that are in subharmonic order. Given a specific martingale $M$, Ghoussoub--Kim--Lim also defined a finest paving of the source space into cells that are invariant under the martingale $M$. 

\smallskip

While in dimension $d = 1$ the decomposition constructed in \cite{GKL19} does not depend on the choice of $\pi \in \MT(\mu,\nu)$ (see the paper \cite{BJ16} by M.\ Beiglb{\"o}ck and N.\ Juillet), it was noted by H.\ De March and N.\ Touzi in \cite{DMT19} that in dimension $d \geqslant 2$ this decomposition does depend on this choice. In \cite{DMT19} and, independently, in the work \cite{OS17} by J.\ Ob{\l}{\'o}j and P.\ Siorpaes, the \textit{finest} such decomposition which works \textit{for all} $\pi \in \MT(\mu,\nu)$ was analyzed. In \cite{DMT19} it was shown that such a universal decomposition exists uniquely (in an almost sure sense). In addition, a martingale transport $\pi^{\dmt} \in \MT(\mu,\nu)$ has been constructed in \cite{DMT19}, which generates this universal decomposition into irreducible sets. While there is no uniqueness in the construction of such \textit{De March--Touzi martingale transports} $\pi^{\dmt} \in \MT(\mu,\nu)$ in \cite{DMT19}, Theorem \ref{gen_case_theom_gen_case_int} above shows, among other features, that the (unique) stretched Brownian motion $\pi^{\sbm} \in \MT(\mu,\nu)$ is such a De March--Touzi martingale transport. In particular, the irreducible pavings into relatively open convex sets induced by $\pi^{\dmt}$ and $\pi^{\sbm}$ coincide (in an almost sure sense made precise below).

\subsection{Definitions and notation}

\begin{itemize}
\item We write $\PP(\Rd)$ for the probability measures on $\Rd$ and $\PP_{p}(\Rd)$ for the subset of probability measures satisfying $\int \vert x \vert^{p} \, d\mu < +\infty$, for $p \in [1,+\infty)$.
\item For $\mu,\nu \in \PP(\Rd)$, we denote by $\Cpl(\mu,\nu)$ the set of all couplings $\pi \in \PP(\Rd \times \Rd)$ between $\mu$ and $\nu$, i.e.\ probability measures $\pi$ on $\Rd \times \Rd$ with first marginal $\mu$ and second marginal $\nu$.
\item We say that $\mu \in \PP_{1}(\Rd)$ is dominated by $\nu \in \PP_{1}(\Rd)$ in convex order and write $\mu \lc \nu$, if for all convex functions $f \colon \Rd \rightarrow \R$ we have $\int f \, d\mu \leqslant \int f \, d\nu$. 
\item For $\mu,\nu \in \PP_{2}(\Rd)$ with $\mu \lc \nu$ we define the collection of martingale transports $\MT(\mu,\nu)$ as those couplings $\pi \in \Cpl(\mu,\nu)$ with barycenter $\bary(\pi_{x}) \coloneqq \int y \, \pi_{x}(dy) = x$, for $\mu$-a.e.\ $x \in \Rd$. Here, the family of probability measures $\{\pi_{x}\}_{x\in\Rd} \subseteq \PP(\Rd)$ is obtained by disintegrating the coupling $\pi$ with respect to its first marginal $\mu$, i.e., $\pi(dx,dy) = \pi_{x}(dy) \, \mu(dx)$.
\item The $d$-dimensional standard Gaussian distribution is denoted by $\gamma$.
\item We denote by $\Cq$ the set of continuous functions $\psi \colon \Rd \rightarrow \R$ with quadratic growth, meaning that there are constants $a,k,\ell \in \R$ with
\[
\ell + \tfrac{\vert \, \cdot \,  \vert^{2}}{2} \leqslant \psi(\, \cdot \,) \leqslant a + k \vert \cdot  \vert^{2}.
\]
We also introduce the set
\begin{equation} \label{def_eq_cqaff}
\Cqaff \coloneqq \big\{ \psi(\, \cdot \,) + \aff(\, \cdot \,) \colon \, \psi \in \Cq, \, \aff \colon \Rd \rightarrow \R \textnormal{ is affine} \big\}.
\end{equation}
\item For two measures $\rho$ and $q$ on $\Rd$ we write $\rho \ast q$ for their convolution. If $f$ is a function, the convolution of $f$ and $q$ is defined as
\[
(f \ast q)(x) \coloneqq \int f(x-y) \, q(dy), \qquad x \in \Rd,
\]
provided $f(x - \, \cdot \,) \in L^{1}(q)$. In particular, $(f \ast \gamma)(x) = \int f(x+y) \, \gamma(dy)$.
\item For a function $f \colon \Rd \rightarrow (-\infty,+\infty]$, its convex conjugate is given by 
\[
f^{\ast}(y) \coloneqq \sup_{x \in \Rd} \big(\langle x,y \rangle - f(x)\big), \qquad y \in \Rd,
\]
and we write
\[
\dom f \coloneqq \{ x \in \Rd \colon f(x) < + \infty \}
\]
for the domain of $f$. 
\item We write $\operatorname{ri}(A)$ and $\overline{A}$ for the relative interior and the closure of a set $A \subseteq \Rd$, respectively. 
\item The support and the closed convex hull of the support of a measure $\rho$ are denoted by $\supp(\rho)$ and $\widehat{\supp}(\rho)$, respectively.
\item The Polish space $\widehat{\mathcal{K}}$ of closed convex subsets of $\Rd$ is equipped with the Wijsman topology \cite{Bee91} and is isomorphically identified with $\operatorname{ri} \widehat{\mathcal{K}} = \{ \operatorname{ri} C \colon C \in \widehat{\mathcal{K}}\}$.
\end{itemize}

\subsection{Relaxed formulation of the dual function} \label{subsec_rfotdf}

Let us return to the technical caveat we made after the statement of Theorem \ref{theorem.1.4.o1a}. For $\psi \in L^{1}(\nu)$, we recall the definition \eqref{def_dual_func_gen} of the dual function $\mathcal{D}(\, \cdot \,)$, which we rewrite as 
\begin{equation} \label{def_dual_func_gen_rel}
\mathcal{D}(\psi) = \int \Big( \int \psi(y) \, \pi_{x}(dy) 
- (\psi^{\ast} \ast \gamma)^{\ast}(x) \Big) \, \mu(dx),
\end{equation}
with $\pi$ being an arbitrary fixed element of $\MT(\mu,\nu)$. As shown in \cite[Proposition 4.2]{BVBST23}, the dual function written in the ``relaxed'' form \eqref{def_dual_func_gen_rel} is well-defined and takes values in the interval $[0,+\infty]$, for a general convex function $\psi \colon \Rd \rightarrow (-\infty,+\infty]$ which is $\mu$-a.s.\ finite. Furthermore, by \cite[Lemma 3.5]{BVBST23} and \cite[Proposition 4.1]{BVBST23}, the value $D(\mu,\nu)$ of the dual problem \eqref{dp_mbb_mcov_dt_f} is equal to
\begin{equation} \label{dual_prob_gen_sec_form}
D(\mu,\nu) = \inf_{\substack{\psi \in \Cqaff, \\ \textnormal{$\psi$ convex}}}\mathcal{D}(\psi) = \inf_{\substack{\mu(\dom \psi) = 1, \\ \textnormal{$\psi$ convex}}}\mathcal{D}(\psi).
\end{equation}
Thus, whenever convenient, by \eqref{dual_prob_gen_sec_form} we are free to choose optimizing sequences $(\psi_{n})_{n \geqslant 1}$ for the dual problem \eqref{dp_mbb_mcov_dt_f} of class $\Cqaff$, as defined in \eqref{def_eq_cqaff} above.

\begin{definition} \label{def_dual_opt_gen_ca} We say that a lower semicontinuous convex function $\psi \colon \Rd \rightarrow (-\infty,+\infty]$ satisfying $\mu(\operatorname{ri}(\dom\psi))=1$ is an optimizer of the dual problem \eqref{dual_prob_gen_sec_form} if $D(\mu,\nu) = \mathcal{D}(\psi)$, for the dual function $\mathcal{D}(\, \cdot \,)$ as defined in \eqref{def_dual_func_gen_rel}. In short, we say that $\psi$ is a dual optimizer.
\end{definition}

We refer to Example 6.7 in \cite{BVBST23} for a case of a dual optimizer $\psi$ in the sense of Definition \ref{def_dual_opt_gen_ca} which is not integrable with respect to $\nu$. 

\section{Equivalence classes of convex functions \texorpdfstring{on $\Rd$}{}} \label{ecocford_sec_three}

In this section we will analyze some phenomena arising for the \textit{equivalence classes $[\psi]$ of convex functions $\psi$ modulo adding affine functions}, as in Definition \ref{def_gc_ecocfeo}.

\begin{definition} \label{def_psi} Let $\psi \colon \Rd \rightarrow \mathds{R}$ be a convex function. For $x \in \Rd$ and $\varepsilon > 0$, we denote by $[\psi]^{x,\varepsilon}$ the equivalence class of convex functions of the form
\begin{equation} \label{eq_def_psi_ii}
\psi^{x,\varepsilon}( \, \cdot \,) = \psi_{n}( \, \cdot \,) + \textnormal{aff}(\, \cdot \,),
\end{equation}
where $\textnormal{aff} \colon \Rd \rightarrow \R$ runs through all affine functions on $\Rd$ such that 
\begin{equation} \label{eq_def_psi_iii}
\psi^{x,\varepsilon}( \, \cdot \,) \geqslant 0
\qquad \textnormal{ and } \qquad
\psi^{x,\varepsilon}(x) < \varepsilon. 
\end{equation}
\end{definition}


\begin{definition} \label{def_sbm_paving} Let $(\psi_{n})_{n \geqslant 1}$ be a sequence of real-valued convex functions on $\Rd$. For $x \in \Rd$, we define the sets
\begin{alignat}{2}
&I^{\textnormal{b}}(x) 
&&\coloneqq \ \Big\{ y \in \Rd \colon \ \lim_{\varepsilon \downarrow 0} \, \sup_{n \geqslant 1} \, \sup_{\psi_{n}^{x,\varepsilon} \in [\psi_{n}]^{x,\varepsilon}} \psi_{n}^{x,\varepsilon}(y) < + \infty\Big\}, 
\label{def_sbm_paving_i} \\
&I^{\textnormal{B}}(x) 
&&\coloneqq \ \operatorname{ri} I^{\textnormal{b}}(x), 
\label{def_sbm_paving_ii} \\
&C^{\textnormal{B}}(x) 
&&\coloneqq \ \overline{I^{\textnormal{b}}(x)}. 
\label{def_sbm_paving_iii}
\end{alignat}
\end{definition}

Here, the superscripts ``\textnormal{b}'' and ``\textnormal{B}'' stand for ``bounded'' and ``Bass'', respectively. The operator $\operatorname{ri}$ denotes the relative interior of a set. We will prove in Proposition \ref{prop_psinbar_bound} below that definition \eqref{eq_fdotbb_int} of $I^{\textnormal{B}}(x)$ coincides with Definition \ref{def_sbm_paving}, thus justifying this abuse of notation.

\begin{proposition} \label{prop_sbm_paving} Let $(\psi_{n})_{n \geqslant 1}$ be a sequence of real-valued convex functions on $\Rd$. The collection $\{I^{\textnormal{B}}(x)\}_{x \in \Rd}$, as introduced in Definition \ref{def_sbm_paving} above, has the following properties:
\begin{enumerate}[label=(\roman*)] 
\item \label{icpodo} $x \in I^{\textnormal{B}}(x)$,
\item \label{icpodpar} for $x_{1} \in I^{\textnormal{B}}(x)$ we have $I^{\textnormal{B}}(x) = I^{\textnormal{B}}(x_{1})$,
\item \label{icpodpartwo} for $x_{1} \notin I^{\textnormal{B}}(x)$ we have $I^{\textnormal{B}}(x) \cap I^{\textnormal{B}}(x_{1}) = \varnothing$.
\end{enumerate}
\end{proposition}

In conclusion, the collection $\{I^{\textnormal{B}}(x)\}_{x \in \Rd}$ is a paving of $\Rd$ into relatively open convex sets.

\begin{proof}[Proof of Proposition \ref{prop_sbm_paving}] \ref{icpodo}: Let $x \in \R^{d}$. We first note that both $I^{\textnormal{b}}(x)$ and $I^{\textnormal{B}}(x)$ are convex sets and by definition the latter set is relatively open. Moreover, it is clear that $x \in I^{\textnormal{b}}(x)$. To show that $x$ is in the relative interior of $I^{\textnormal{b}}(x)$, we suppose that $x \notin I^{\textnormal{B}}(x)$, i.e., that $x$ is a relative boundary point of $I^{\textnormal{b}}(x)$. By \cite[Corollary 11.6.2]{Roc70} there exists a functional $y^{\ast} \in \Rd$ and some $\bar{x} \in I^{\textnormal{b}}(x)$ such that
\begin{equation} \label{icpodoa}
\langle \bar{x}, y^{\ast} \rangle < \max_{y \in I^{\textnormal{b}}(x)} \langle y, y^{\ast} \rangle = \langle x, y^{\ast} \rangle.
\end{equation}
For notational convenience and without loss of generality we assume that $\vert x - \bar{x} \vert = 1$. We will show that
\begin{equation} \label{icpodoa_f}
\lim_{\varepsilon \downarrow 0} \, \sup_{n \geqslant 1} \, \sup_{\psi_{n}^{x,\varepsilon} \in [\psi_{n}]^{x,\varepsilon}}  \psi_{n}^{x,\varepsilon}(\bar{x}) = + \infty,
\end{equation}
so that $\bar{x} \notin I^{\textnormal{b}}(x)$, which will yield the desired contradiction. 

\smallskip

We draw a straight line through $\bar{x}$ and $x$, i.e., we consider the function
\begin{equation} \label{icpodoa_c}
t \longmapsto x + t(x - \bar{x}), \qquad t \in \mathds{R}.
\end{equation}
For $\varepsilon > 0$, let us consider the point $y = y(\varepsilon) = x + \varepsilon^{2}(x-\bar{x})$ on this line. By \eqref{icpodoa} the point $y$ cannot be an element of $I^{\textnormal{b}}(x)$. Thus, for $\varepsilon \in (0,1)$ we can find $n \geqslant 1$ and a function $\psi_{n}^{x,\varepsilon} \in [\psi_{n}]^{x,\varepsilon}$ such that
\begin{equation} \label{icpodoa_b}
\psi_{n}^{x,\varepsilon}(x+\varepsilon^{2}(x - \bar{x})) - \psi_{n}^{x,\varepsilon}(x) \geqslant 1.
\end{equation}
By elementary geometry we observe that the convex function $\psi_{n}^{x,\varepsilon}$ in \eqref{icpodoa_b} has the following property: the affine function $\textnormal{af}(\, \cdot \,)$ on the real line given by
\begin{equation} \label{icpodoa_e}
t \longmapsto \textnormal{af}(t) \coloneqq \tfrac{1}{\varepsilon} t - \varepsilon, \qquad t \in \mathds{R}
\end{equation}
is less than or equal to the restriction of $\psi_{n}^{x,\varepsilon}$ to the line \eqref{icpodoa_c}, i.e., the function
\[
t \longmapsto \psi_{n}^{x,\varepsilon}(x+t(x - \bar{x})), \qquad t \in \mathds{R}.
\]
By Hahn--Banach we can extend the function $x+t(x - \bar{x}) \mapsto \textnormal{af}(t)$ to an affine function $\textnormal{aff}(\, \cdot \,)$ defined on $\R^{d}$ such that $\textnormal{aff}(y) \leqslant \psi_{n}^{x,\varepsilon}(y)$ for all $y \in \Rd$. Note that we have $\textnormal{aff}(x) = \textnormal{af}(0) = - \varepsilon$. Therefore the function
\[
y \longmapsto \psi_{n}^{x,2\varepsilon}(y) \coloneqq \psi_{n}^{x,\varepsilon}(y) - \textnormal{aff}(y), 
\qquad y \in \Rd
\]
is a representative of $[\psi_{n}]^{x,2\varepsilon}$. At the same time we have $\textnormal{aff}(\bar{x}) = \textnormal{af}(-1) = - \tfrac{1}{\varepsilon}  - \varepsilon$ and thus
\[
\psi_{n}^{x,2\varepsilon}(\bar{x}) = \psi_{n}^{x,\varepsilon}(\bar{x}) - \textnormal{aff}(\bar{x}) \geqslant \tfrac{1}{\varepsilon}  + \varepsilon \geqslant \tfrac{1}{\varepsilon}.
\]
This contradicts the assumption $\bar{x} \in I^{\textnormal{b}}(x)$ and finishes the proof of \ref{icpodo}.
 
\medskip

\noindent \ref{icpodpar}: Let $x \in \Rd$. We first show the following claim: \textit{For every compact subset $K \subseteq I^{\textnormal{B}}(x)$ there is $\bar{M} > 0$ and $\bar{\varepsilon} \in (0,1]$ such that for each $n \in \mathds{N}$ and $\psi_{n}^{x,\bar{\varepsilon}} \in [\psi_{n}]^{x,\bar{\varepsilon}}$ we have
\begin{equation} \label{icpodpar_cl_one}
K \subseteq \{ \psi_{n}^{x,\bar{\varepsilon}} \leqslant \bar{M} \}.
\end{equation}}
For the proof of the claim, we define the set
\[
B_{M, \varepsilon} \coloneqq \bigcap_{n \in \mathds{N}} \bigcap_{\psi_{n}^{x,\varepsilon} \in [\psi_{n}]^{x,\varepsilon}} \{ \psi_{n}^{x,\varepsilon} \leqslant M \}.
\]
By Definition \ref{def_sbm_paving}, we have
\[
\bigcup_{M > 0} \bigcup_{\varepsilon > 0} B_{M, \varepsilon} 
= I^{\textnormal{b}}(x) \supseteq I^{\textnormal{B}}(x).
\]
We denote by $\Aff(x)$ the affine span of $I^{\textnormal{B}}(x)$ and by $B_{M, \varepsilon}^{\circ}$ the relative interior of $B_{M, \varepsilon}$ in $\Aff(x)$. Then for every compact subset $K \subseteq I^{\textnormal{B}}(x)$ we have
\[
\bigcup_{M > 0} \bigcup_{\varepsilon > 0} B_{M, \varepsilon}^{\circ} 
= I^{\textnormal{B}}(x) \supseteq K.
\]
By compactness there is some $\bar{M} > 0$ and $\bar{\varepsilon} > 0$ such that
\[
B_{\bar{M},\bar{\varepsilon}}^{\circ} \supseteq K,
\]
which readily implies \eqref{icpodpar_cl_one}.

\smallskip

Let $x_{1} \in I^{\textnormal{B}}(x)$ and choose $\delta > 0$ such that 
\[
K_{\delta} \coloneqq \overline{B}_{\delta}(x_{1}) \cap \Aff(x) \subseteq I^{\textnormal{B}}(x).
\]
Take $\bar{M} > 0$ and $\bar{\varepsilon} \in (0,1]$ such that \eqref{icpodpar_cl_one} is satisfied for $K = K_{\delta}$. Now we make the following claim: \textit{There exists a constant $L > 0$ such that, for every $\varepsilon \in (0,\bar{\varepsilon}]$ and every choice of representatives $\psi_{n}^{x,\varepsilon} \in [\psi_{n}]^{x,\varepsilon}$ and $\psi_{n}^{x_{1},\varepsilon} \in [\psi_{n}]^{x_{1},\varepsilon}$, the affine functions
\[
\aff_{n} \coloneqq \psi_{n}^{x,\varepsilon} - \psi_{n}^{x_{1},\varepsilon}
\]
satisfy the Lipschitz estimate}
\begin{equation} \label{icpodpar_lip}
\forall y,z \in \Aff(x) \colon \ 
\vert \aff_{n}(y) - \aff_{n}(z) \vert \leqslant \tfrac{L}{\delta} \vert y - z \vert.
\end{equation}
Negating the claim, we suppose that, for every $L > 0$, there is a point $x_{2} \in B_{\delta}(x_{1}) \cap \Aff(x)$ and $n \in \mathds{N}$ such that
\[
\aff_{n}(x_{2}) - \aff_{n}(x_{1}) \geqslant L.
\]
We then obtain
\[
\psi_{n}^{x,\varepsilon}(x_{2}) - \psi_{n}^{x,\varepsilon}(x_{1})
\geqslant \psi_{n}^{x_{1},\varepsilon}(x_{2}) - \psi_{n}^{x_{1},\varepsilon}(x_{1}) + L.
\]
Using the estimates $\psi_{n}^{x,\varepsilon}(x_{1}) \geqslant 0$, $\psi_{n}^{x_{1},\varepsilon}(x_{2}) \geqslant 0$, and $\psi_{n}^{x_{1},\varepsilon}(x_{1}) \leqslant 1$, we conclude that
\[
\psi_{n}^{x,\varepsilon}(x_{2}) \geqslant L-1,
\]
which is a contradiction to $\psi_{n}^{x,\varepsilon}(x_{2}) \leqslant \bar{M}$, thus finishing the proof of \eqref{icpodpar_lip}. 

\smallskip

Finally, it follows from \eqref{icpodpar_lip} that
\begin{equation} \label{icpodpar_pre}
\forall x_{1} \in I^{\textnormal{B}}(x) \colon \
I^{\textnormal{b}}(x) \cap \Aff(x) = I^{\textnormal{b}}(x_{1}) \cap \Aff(x).
\end{equation}
To obtain the full assertion \ref{icpodpar} from \eqref{icpodpar_pre}, we still have to show that $I^{\textnormal{b}}(x_{1})$, and therefore also $I^{\textnormal{B}}(x_{1})$, is contained in the affine space $\Aff(x)$. If this were not the case, the affine space $\Aff(x_{1})$ spanned by $I^{\textnormal{B}}(x_{1})$ would strictly contain $\Aff(x)$. Reversing the roles of $x$ and $x_{1}$, and applying \eqref{icpodpar_lip} to $\Aff(x_{1})$ rather than $\Aff(x)$, we arrive at the desired contradiction, which completes the proof of \ref{icpodpar}. 

\medskip

\noindent \ref{icpodpartwo}: The final statement \ref{icpodpartwo} is a consequence of \ref{icpodo} and \ref{icpodpar}.
\end{proof}

We next show a variant of Proposition \ref{prop_sbm_paving}, where we do not have to refer to $\varepsilon > 0$ as in Definitions \ref{def_psi} and \ref{def_sbm_paving} above.

\begin{proposition} \label{prop_psinbar_bound} Let $(\psi_{n})_{n \geqslant 1}$ be a sequence of $(-\infty,+\infty]$-valued convex functions on $\Rd$. Fix representatives $\bar{\psi}_{n} \in [\psi_{n}]$, $x \in \Rd$, and $I^{\textnormal{B}}(x)$ as defined in Definition \ref{def_sbm_paving}. The sequence $(\bar{\psi}_{n}(y))_{n \geqslant 1}$ is bounded for some $y \in I^{\textnormal{B}}(x)$ if and only if it is bounded for all $y \in I^{\textnormal{B}}(x)$. In this case, $(\bar{\psi}_{n})_{n \geqslant 1}$ is uniformly bounded on compact subsets of $I^{\textnormal{B}}(x)$.

\smallskip

In particular, the definition \eqref{eq_fdotbb_int} of $I^{\textnormal{B}}(x)$ coincides with Definition \ref{def_sbm_paving}.
\begin{proof} Let $(\varepsilon_{n})_{n \geqslant 1}$ be a null sequence and take $\psi_{n}^{x,\varepsilon_{n}} \in [\psi_{n}]^{x,\varepsilon_{n}}$. By definition \eqref{def_sbm_paving_i}, \eqref{def_sbm_paving_ii} of $I^{\textnormal{B}}(x)$, we have that $(\psi_{n}^{x,\varepsilon_{n}}(y))_{n \geqslant 1}$ is bounded, for all $y \in I^{\textnormal{B}}(x)$. For arbitrary representatives $\bar{\psi}_{n} \in [\psi_{n}]$ we have that $\textnormal{aff}_{n} \coloneqq \bar{\psi}_{n} - \psi_{n}^{x,\varepsilon_{n}}$ is an affine function on $\Rd$. As $\bar{\psi}_{n} \geqslant 0$, we have that $(\textnormal{aff}_{n}(y))_{n \geqslant 1}$ is bounded from below, for every $y \in I^{\textnormal{B}}(x)$. If we also have that $(\bar{\psi}_{n}(y_{0}))_{n \geqslant 1}$ is bounded, for some $y_{0} \in I^{\textnormal{B}}(x)$, we obtain that $(\textnormal{aff}_{n}(y_{0}))_{n \geqslant 1}$ is bounded. This implies that $(\textnormal{aff}_{n}(y))_{n \geqslant 1}$ is bounded, for all $y$ in the affine span of $I^{\textnormal{B}}(x)$. In particular, $(\bar{\psi}_{n}(y))_{n \geqslant 1}$ is bounded, for all $y \in I^{\textnormal{B}}(x)$. By convexity, $(\bar{\psi}_{n})_{n \geqslant 1}$ is uniformly bounded on compact subsets of $I^{\textnormal{B}}(x)$.

\smallskip

As regards the final assertion, the set $I^{\textnormal{B}}(x)$ defined in \eqref{eq_fdotbb_int} obviously is contained in the set $I^{\textnormal{B}}(x)$ defined in \eqref{def_sbm_paving_i}, \eqref{def_sbm_paving_ii}. For the reverse inclusion, suppose there is $y \in \Rd$ and a sequence $(\psi_{n}^{x,\varepsilon_{n}})_{n \geqslant 1}$ such that $(\psi_{n}^{x,\varepsilon_{n}}(y))_{n \geqslant 1}$ is unbounded, while, for every choice of representatives $\bar{\psi}_{n} \in [\psi_{n}]$ with $\bar{\psi}_{n}(x) \leqslant 1$, the sequence $(\bar{\psi}_{n}(y))_{n \geqslant 1}$ is bounded. This is a contradiction to the first part of Proposition \ref{prop_psinbar_bound}, finishing the proof of the final assertion.
\end{proof}
\end{proposition}

So far, the sequence $(\psi_{n})_{n \geqslant 1}$ of real-valued convex functions on $\Rd$ had nothing to do with the pair $(\mu,\nu)$. We now impose the additional assumption that $(\psi_{n})_{n \geqslant 1} \subseteq \Cqaff$ is an optimizing sequence of convex functions for the dual problem \eqref{dp_mbb_mcov_dt_f}. In this way, we can relate the support of $\nu$ to $I^{\textnormal{B}}(x)$.

\begin{proposition} \label{prop_psinbar_bound_sharp} Let $\mu, \nu \in \PP_{2}(\Rd)$ with $\mu \lc \nu$. Let $(\psi_{n})_{n \geqslant 1} \subseteq \Cqaff$ be an optimizing sequence of convex functions for the dual problem \eqref{dp_mbb_mcov_dt_f}, with induced Bass paving $\{I^{\textnormal{B}}(x)\}_{x \in \Rd}$. Suppose that there is $x \in \Rd$ such that, for $I^{\textnormal{B}} \coloneqq I^{\textnormal{B}}(x)$, we have $\mu(I^{\textnormal{B}}) = 1$. Then for $C^{\textnormal{B}} = \overline{I^{\textnormal{B}}}$ we have
\begin{equation} \label{lem_cruc_incl_i_sharp}
\widehat{\supp}(\nu) \subseteq C^{\textnormal{B}}.
\end{equation}
\begin{proof} Applying \cite[Lemma 7.9]{BVBST23} gives
\begin{equation} \label{lem_cruc_incl_ii_a_sharp}
a \coloneqq \sup_{n \geqslant 1} \int_{x \in I^{\textnormal{B}}} \int_{y \in \Rd} \big( \psi_{n}(y) - \psi_{n}(x) \big) \, \pi^{\sbm}_{x}(dy) \, \mu(dx) < +\infty.
\end{equation}

\smallskip

Arguing by contradiction, we assume that \eqref{lem_cruc_incl_i_sharp} is not true. Then we can find an open half space $H$ with $H \cap C^{\textnormal{B}} = \varnothing$ such that $\nu(H) > 0$. By Proposition \ref{prop_psinbar_bound}, we can choose representatives $\bar{\psi}_{n} \in [\psi_{n}]$ such that the sequence $(\bar{\psi}_{n})_{n \geqslant 1}$ is pointwise bounded on $I^{\textnormal{B}}$. We denote by $\tilde{\psi}_{n} \in L^{1}(\nu)$ the function which coincides with $\bar{\psi}_{n}$ on $C^{\textnormal{B}}$, and assumes the value $+\infty$ outside of $C^{\textnormal{B}}$. As observed after Definition \ref{def_gc_ecocfeo}, the sequence $(\tilde{\psi}_{n})_{n \geqslant 1}$ still is an optimizing sequence of lower semicontinuous convex functions in $L^{1}(\nu)$. It also (trivially) satisfies
\begin{equation} \label{prop_psinbar_bound_i_sharp}
\lim_{n \rightarrow \infty} \tilde{\psi}_{n}(y) = + \infty,
\qquad y \in H.
\end{equation}
Since $\nu(H) > 0$, we can find a compact subset $K \subseteq I^{\textnormal{B}}$ such that
\begin{equation} \label{prop_psinbar_conv_pos_iii_sharp}
\mu\big(\big\{ x \in K \colon \pi^{\sbm}_{x}(H) > 0 \big\}\big) > 0.
\end{equation}
As a consequence of \eqref{prop_psinbar_conv_pos_iii_sharp}, the set $K \times H$ has positive $\pi^{\sbm}$-measure.

\smallskip

Now observe that \eqref{lem_cruc_incl_ii_a_sharp} also holds when we replace $\psi_{n}$ by $\tilde{\psi}_{n}$. Furthermore, as
\[
\int \big( \tilde{\psi}_{n}(y) - \tilde{\psi}_{n}(x) \big) \, \pi^{\sbm}_{x}(dy) \geqslant 0,
\]
for $\mu$-a.e.\ $x \in \Rd$, we conclude that
\[
a \geqslant \sup_{n \geqslant 1} \int_{x \in K} \int_{y \in \Rd} \big( \tilde{\psi}_{n}(y) - \tilde{\psi}_{n}(x) \big) \, \pi^{\sbm}_{x}(dy) \, \mu(dx).
\]
As $(\tilde{\psi}_{n})_{n \geqslant 1}$ is uniformly bounded on compact subsets of $I^{\textnormal{B}}$, we have
\[
m \coloneqq \sup_{n \geqslant 1}  \sup_{x \in K}  \tilde{\psi}_{n}(x) < + \infty.
\]
Consequently,
\[
a \geqslant \sup_{n \geqslant 1} \int_{x \in K} \int_{y \in \Rd}  \tilde{\psi}_{n}(y) \, \pi^{\sbm}_{x}(dy) \, \mu(dx)  - m \mu(K).
\]
Using the fact that $\tilde{\psi}_{n} \geqslant 0$, we have
\begin{equation} \label{prop_psinbar_conv_pos_vi_sharp}
a \geqslant \sup_{n \geqslant 1} \int_{y \in \Rd} \int_{x \in K}  \tilde{\psi}_{n}(y) \, \pi^{\sbm}_{x}\big\vert_{H}(dy) \, \mu(dx)  - m \mu(K).
\end{equation}
Next we define a probability measure $\tilde{\nu}$ on $\Rd$ by
\[
\tilde{\nu}(\, \cdot \,) \coloneqq \tfrac{1}{c} \int_{K} \pi^{\sbm}_{x}\big\vert_{H}(\, \cdot \,) \, \mu(dx), 
\]
with normalizing constant
\[
c \coloneqq \int_{K} \pi^{\sbm}_{x}(H) \, \mu(dx). 
\]
Note that $c > 0$ since $\pi^{\sbm}(K \times H) > 0$. In terms of $\tilde{\nu}$, we can express the right-hand side of \eqref{prop_psinbar_conv_pos_vi_sharp} as 
\[
c \sup_{n \geqslant 1} \int_{\Rd} \tilde{\psi}_{n}(y) \, \tilde{\nu}(dy) - m \mu(K).
\]
Applying Jensen's inequality yields
\begin{equation} \label{prop_psinbar_conv_pos_iv_contr_sharp}
a \geqslant 
c \sup_{n \geqslant 1} \tilde{\psi}_{n}(\tilde{y}) - m \mu(K),
\end{equation}
where
\[
\tilde{y} \coloneqq \bary(\tilde{\nu}) = \int_{\Rd} y \, \tilde{\nu}(dy)
\]
denotes the barycenter of the probability measure $\tilde{\nu}$. Observing that $\tilde{y} \in H$ and using \eqref{prop_psinbar_bound_i_sharp}, the right-hand side of \eqref{prop_psinbar_conv_pos_iv_contr_sharp} is equal to $+\infty$. This is a contradiction to the finiteness of $a$ and shows the inclusion \eqref{lem_cruc_incl_i_sharp}.
\end{proof}
\end{proposition}

Combining Proposition \ref{prop_psinbar_bound_sharp} with \cite[Proposition 7.20]{BVBST23}, we obtain the following result.

\begin{proposition} \label{prop_psinbar_bound_sharper} Let $\mu, \nu \in \PP_{2}(\Rd)$ with $\mu \lc \nu$. Let $(\psi_{n})_{n \geqslant 1} \subseteq \Cqaff$ be an optimizing sequence of convex functions for the dual problem \eqref{dp_mbb_mcov_dt_f}. Suppose that there is $x \in \Rd$ such that for $I^{\textnormal{B}} \coloneqq I^{\textnormal{B}}(x)$ we have $\mu(I^{\textnormal{B}}) = 1$. Then there exist representatives $\bar{\psi}_{n} \in [\psi_{n}]$ and a lower semicontinuous convex function $\psilim \colon \Rd \rightarrow [0,+\infty]$ such that
\begin{alignat}{2}
\lim_{n \rightarrow \infty} \bar{\psi}_{n}(y) &= \psilim(y) < + \infty, \qquad &&y \in I^{\textnormal{B}}, \label{prop_psinbar_bound_sharper_ii} \\
\lim_{n \rightarrow \infty} \bar{\psi}_{n}(y) &=\psilim(y) =  + \infty, \qquad &&y \in \Rd \setminus C^{\textnormal{B}}. \label{prop_psinbar_bound_sharper_iii}
\end{alignat}
Furthermore, 
\begin{equation} \label{prop_psinbar_bound_sharper_i}
\widehat{\supp}(\nu) = C^{\textnormal{B}},
\end{equation}
and the function $\psilim$ is a dual optimizer for the pair $(\mu,\nu)$.
\begin{proof} Note that $\mu(I^{\textnormal{B}}) = 1$ by assumption and $\widehat{\supp}(\nu) \subseteq C^{\textnormal{B}}$ by Proposition \ref{prop_psinbar_bound_sharp}. Furthermore, by Proposition \ref{prop_psinbar_bound}, we can choose representatives $\tilde{\psi}_{n} \in [\psi_{n}]$ such that the sequence $(\tilde{\psi}_{n})_{n \geqslant 1}$ is pointwise bounded on $I^{\textnormal{B}}$, i.e.
\[
\forall y \in I^{\textnormal{B}} \colon \ \sup_{n \geqslant 1} \tilde{\psi}_{n}(y) < + \infty.
\]
By \cite[Proposition 7.20]{BVBST23}, we conclude that there are representatives $\bar{\psi}_{n} \in [\tilde{\psi}_{n}] = [\psi_{n}]$ such that \eqref{prop_psinbar_bound_sharper_ii} and \eqref{prop_psinbar_bound_sharper_iii} are satisfied; moreover, we have the identity \eqref{prop_psinbar_bound_sharper_i}, and $\psilim$ is a dual optimizer.
\end{proof}
\end{proposition}

In order to apply the disintegration theorem, we will need the following measurability property of the Bass paving $\{I^{\textnormal{B}}(x)\}_{x \in \Rd}$.

\begin{proposition} \label{sbm_map_prop} Let $(\psi_{n})_{n \geqslant 1} \subseteq \Cqaff$ be a sequence of convex functions. The map $C^{\textnormal{B}} \colon \Rd \rightarrow \widehat{\mathcal{K}}$ given by 
\begin{equation} \label{sbm_map_prop_one}
x \longmapsto C^{\textnormal{B}}(x) = \overline{I^{\textnormal{B}}(x)},
\end{equation}
where $\widehat{\mathcal{K}}$ is equipped with the Wijsman topology, is Borel measurable.
\begin{proof} As noted in \cite{DMT19}, it suffices to show that the set
\[
B^{V} \coloneqq \big\{ x \in \Rd \colon C^{\textnormal{B}}(x) \cap V \neq \varnothing \big\}
\]
is Borel measurable, for every open set $V \subseteq \Rd$. To see this, we fix an open set $V \subseteq \Rd$ and note that
\[
B^{V} = \big\{ x \in \Rd \colon I^{\textnormal{b}}(x) \cap V \neq \varnothing \big\}.
\]
Consequently,
\[
B^{V} = 
\bigcup_{\varepsilon > 0} \,
\bigcup_{M > 0} \,
\bigcap_{n \in \mathds{N}}  
\bigcap_{\substack{\bar{\psi}_{n} \in [\psi_{n}], \\ \{ \bar{\psi}_{n} < M\} \cap V \neq \varnothing}}  \big\{ x \in \Rd \colon \bar{\psi}_{n}(x) < \varepsilon \big\}.
\]
Hence $B^{V}$ is Borel measurable if we can show that the last intersection can be replaced by a countable intersection. In order to do so, we fix a basis $(e_{j})_{j=1}^{d}$ of $\Rd$ as well as representatives $\bar{\psi}_{n} \in [\psi_{n}]$, and define the equivalence class $[\psi_{n}]^{\mathds{Q}}$ to consist of all functions $\tilde{\psi}_{n} \in [\psi_{n}]$ such that the affine function $\bar{\psi}_{n} - \tilde{\psi}_{n}$ has rational coefficients when expressed in the basis $(e_{j})_{j=1}^{d}$. We then have to show that
\[
U \coloneqq \bigcap_{\substack{\bar{\psi}_{n} \in [\psi_{n}], \\ \{ \bar{\psi}_{n} < M\} \cap V \neq \varnothing}}  \big\{ x \in \Rd \colon \bar{\psi}_{n}(x) < \varepsilon \big\}
= 
\bigcap_{\substack{\tilde{\psi}_{n} \in [\psi_{n}]^{\mathds{Q}}, \\ \{ \tilde{\psi}_{n} < M\} \cap V \neq \varnothing}}  \big\{ x \in \Rd \colon \tilde{\psi}_{n}(x) < \varepsilon \big\}
 \eqqcolon U^{\mathds{Q}}.
\]
Clearly $U \subseteq U^{\mathds{Q}}$, so that it remains to show the inclusion $\Rd \setminus U \subseteq \Rd \setminus U^{\mathds{Q}}$. Fix $x \in \Rd$ and $\bar{\psi}_{n} \in [\psi_{n}]$ such that $\{ \bar{\psi}_{n} < M\} \cap V  \neq \varnothing$ and $\{\bar{\psi}_{n}(x) \geqslant \varepsilon \}$, i.e., $x \in \Rd \setminus U$. Choose $\eta > 0$ such that we still have
\[
\{ \bar{\psi}_{n} + \eta < M\} \cap V  \neq \varnothing.
\]
Note that $\bar{\psi}_{n} + \eta$ is bounded from below by $\eta$ and that the function $\bar{\psi}_{n}$ is in $\Cqaff$, so that $\bar{\psi}_{n}(y) \geqslant \frac{\vert y \vert^{2}}{4}$ for $y$ outside a bounded set. Hence we can find a representative $\tilde{\psi}_{n} \in [\psi_{n}]^{\mathds{Q}}$ such that 
\[
\tilde{\psi}_{n} \geqslant 0, \qquad
\{ \tilde{\psi}_{n} < M\} \cap V  \neq \varnothing 
\qquad \textnormal{ and } \qquad
\{\tilde{\psi}_{n}(x) \geqslant \varepsilon \}.
\]
Therefore $x \in \Rd \setminus U^{\mathds{Q}}$. This shows the claim $U = U^{\mathds{Q}}$ and we conclude that the set $B^{V}$ is Borel measurable.
\end{proof}
\end{proposition}

Having established the Borel measurability of the map $x \mapsto C^{\textnormal{B}}(x)$, we can disintegrate the probability measure $\mu$ with respect to this mapping. 

\begin{lemma} \label{lem_dis_theo} Let $(\psi_{n})_{n \geqslant 1} \subseteq \Cqaff$ be a sequence of convex functions. Define $\kappa^{\textnormal{B}} \coloneqq C^{\textnormal{B}}(\mu)$ to be the pushforward measure of $\mu$ under the map $C^{\textnormal{B}} \colon \Rd \rightarrow \widehat{\mathcal{K}}$ of \eqref{sbm_map_prop_one}, which induces a probability measure on $\widehat{\mathcal{K}}$ (recall that we identify $\widehat{\mathcal{K}}$ with $\operatorname{ri} \widehat{\mathcal{K}}$). There exists a $\kappa^{\textnormal{B}}$-a.s.\ unique family of probability measures $\{\mu^{I^{\textnormal{B}}}\}_{I^{\textnormal{B}} \in \, \widehat{\mathcal{K}}} \subseteq \PP_{2}(\Rd)$ such that $\mu^{I^{\textnormal{B}}}(I^{\textnormal{B}}) = 1$, for $\kappa^{\textnormal{B}}$-a.e.\ $I^{\textnormal{B}} \in \widehat{\mathcal{K}}$, and 
\begin{equation} \label{lem_dis_theo_eq_mu}
\mu(\, \cdot \,) 
= \int_{\widehat{\mathcal{K}}} \, \mu^{I^{\textnormal{B}}}(\, \cdot \,) \ \kappa^{\textnormal{B}}(dI^{\textnormal{B}}).
\end{equation}
\begin{proof} The statement follows from applying the disintegration theorem (see, for instance, \cite[Theorem 5.3.1]{AGS08} or \cite[Chapter III, p.\ 78]{DM78}), which is possible thanks to the Borel measurability of the function \eqref{sbm_map_prop_one}.
\end{proof}
\end{lemma}

Note that Lemma \ref{lem_dis_theo} rigorously justifies the decomposition \eqref{eq_gen_case_deco_mu_fit} as already stated in Section \ref{sec_int_gen_case}.

\smallskip

We refer to Examples \ref{ex_two_bm_gen_case}  and \ref{example_circles} below for a concrete visualization of Lemma \ref{lem_dis_theo}. For example, the measure $\kappa^{\textnormal{B}}$ in Example \ref{ex_two_bm_gen_case}, considered as a probability measure on $\widehat{\mathcal{K}}$, gives probabilities $\frac{1}{2}$ to $\R_{-} = (-\infty,0)$ and $\R_{+} = (0,+\infty)$, as well as probability $0$ to the singleton $\{0\}$. The Bass paving consists of the relatively open sets $\R_{-}$, $\{0\}$, $\R_{+}$, where $\{0\}$ has measure zero with respect to $\kappa^{\textnormal{B}}$ and may therefore be neglected. The probability measures $\mu^{\R_{-}}$ and $\mu^{\R_{+}}$ are the normalized restrictions of $\mu$ to $\R_{-}$ and $\R_{+}$, respectively.

\section{The decomposition of the primal problem}

The disintegration result of Lemma \ref{lem_dis_theo} allows us to decompose the primal problem \eqref{pp_mbb_mcov_dt_f} from $\mu$ to $\nu$ into a family of \textit{local} primal problems, defined on $\kappa^{\textnormal{B}}$-a.e.\ set $I^{\textnormal{B}} \in \widehat{\mathcal{K}}$. To this end, we define (as in \eqref{eq_gen_case_deco_nu_fit} above) the $\kappa^{\textnormal{B}}$-a.s.\ unique family of probability measures $\{\nu^{I^{\textnormal{B}}}\}_{I^{\textnormal{B}} \in \, \widehat{\mathcal{K}}} \subseteq \PP_{2}(\Rd)$ by 
\begin{equation} \label{eq_dis_nu_sbm_single}
\nu^{I^{\textnormal{B}}}(\, \cdot \,) \coloneqq 
\int_{I^{\textnormal{B}}} \, \pi_{x}^{\sbm}(\, \cdot \,) \ \mu^{I^{\textnormal{B}}}(dx),
\end{equation}
where $\pi^{\sbm}(dx,dy) = \pi_{x}^{\sbm}(dy) \, \mu(dx) \in \MT(\mu,\nu)$ is the optimizer of the primal problem \eqref{pp_mbb_mcov_dt_f}. As noted in \eqref{eq_gen_case_deco_nu_fit_decomp}, the family $\{\nu^{I^{\textnormal{B}}}\}_{I^{\textnormal{B}} \in \, \widehat{\mathcal{K}}}$ then is a decomposition of $\nu$, i.e.
\begin{equation} \label{eq_dis_nu_sbm}
\nu(\, \cdot \,) = 
\int_{\widehat{\mathcal{K}}} \, \nu^{I^{\textnormal{B}}}(\, \cdot \,) \ \kappa^{\textnormal{B}}(dI^{\textnormal{B}}).
\end{equation}
By construction we have that $\mu^{I^{\textnormal{B}}} \lc \nu^{I^{\textnormal{B}}}$, for $\kappa^{\textnormal{B}}$-a.e.\ $I^{\textnormal{B}} \in \widehat{\mathcal{K}}$. We obtain the following decomposition of the primal problem \eqref{pp_mbb_mcov_dt_f}.

\begin{lemma}[\textsc{Decomposition of the primal problem}] \label{lem_dotpsbmp} For $\kappa^{\textnormal{B}}$-a.e.\ $I^{\textnormal{B}} \in \widehat{\mathcal{K}}$, the optimizer of the local primal problem
\[
P(\mu^{I^{\textnormal{B}}},\nu^{I^{\textnormal{B}}}) 
= \sup_{\pi \in \MT(\mu^{I^{\textnormal{B}}},\nu^{I^{\textnormal{B}}})} \int \MCov(\pi_{x},\gamma) \, \mu^{I^{\textnormal{B}}}(dx)
\]
from $\mu^{I^{\textnormal{B}}}$ to $\nu^{I^{\textnormal{B}}}$ equals
\begin{equation} \label{eq_def_pisbm_loc_second}
\pi^{I^{\textnormal{B}}}(dx,dy) 
\coloneqq \pi_{x}^{\sbm}(dy) \, \mu^{I^{\textnormal{B}}}(dx) \in \MT(\mu^{I^{\textnormal{B}}},\nu^{I^{\textnormal{B}}}).
\end{equation}
In particular, 
\begin{align}
P(\mu,\nu)
&= \sup_{\pi \in \MT(\mu,\nu)} \int \MCov(\pi_{x},\gamma) \, \mu(dx) 
= \int \MCov(\pi_{x}^{\sbm},\gamma) \, \mu(dx)  \label{lem_dotpsbmp_i_i} \\
&= \int_{\widehat{\mathcal{K}}} \, \bigg(\int \MCov(\pi_{x}^{\sbm},\gamma) \, \mu^{I^{\textnormal{B}}}(dx) \bigg) \, \kappa^{\textnormal{B}}(dI^{\textnormal{B}}) \label{lem_dotpsbmp_i_ii} \\
&= \int_{\widehat{\mathcal{K}}} \, P(\mu^{I^{\textnormal{B}}},\nu^{I^{\textnormal{B}}}) \, \kappa^{\textnormal{B}}(dI^{\textnormal{B}}).\label{lem_dotpsbmp_i_iii} 
\end{align}
\end{lemma}

In other words, the martingale transport $\pi^{I^{\textnormal{B}}} \in \MT(\mu^{I^{\textnormal{B}}},\nu^{I^{\textnormal{B}}})$ of \eqref{eq_def_pisbm_loc_second} above is the stretched Brownian motion from $\mu^{I^{\textnormal{B}}}$ to $\nu^{I^{\textnormal{B}}}$.

\begin{proof}[Proof of Lemma \ref{lem_dotpsbmp}] We note that \eqref{lem_dotpsbmp_i_i} and \eqref{lem_dotpsbmp_i_ii} follow from the definition of stretched Brownian motion and \eqref{lem_dis_theo_eq_mu}, respectively. Therefore we only have to show the equality \eqref{lem_dotpsbmp_i_iii}. By contradiction, suppose there is a measurable set $\mathcal{B} \subseteq \widehat{\mathcal{K}}$ with $\kappa^{\textnormal{B}}(\mathcal{B}) > 0$ and a $\kappa^{\textnormal{B}}$-measurable function 
\[
\mathcal{B} \ni I^{\textnormal{B}} \longmapsto \bar{\pi}^{I^{\textnormal{B}}} \in \MT(\mu^{I^{\textnormal{B}}},\nu^{I^{\textnormal{B}}})
\]
such that
\[
\int \MCov(\bar{\pi}_{x}^{I^{\sbm}},\gamma) \, \mu^{I^{\textnormal{B}}}(dx)
> \int \MCov(\pi_{x}^{I^{\textnormal{B}}},\gamma) \, \mu^{I^{\textnormal{B}}}(dx).
\]
Now for $x \in \Rd$ define
\[ 
\hat{\pi}_{x}^{I^{\textnormal{B}}} \coloneqq
\begin{cases}
\bar{\pi}_{x}^{I^{\textnormal{B}}}, & I^{\textnormal{B}} \in \mathcal{B}, \\
\pi_{x}^{\sbm},& I^{\textnormal{B}} \in \widehat{\mathcal{K}} \setminus \mathcal{B}
\end{cases}
\]
and
\[
\hat{\pi}(dx,dy) \coloneqq \hat{\pi}_{x}^{I^{\textnormal{B}}}(dy) \, \mu^{I^{\textnormal{B}}}(dx) \, \kappa^{\textnormal{B}}(dI^{\textnormal{B}}) \in \MT(\mu,\nu).
\]
But then for the primal problem from $\mu$ to $\nu$ we have
\[
\int \MCov(\hat{\pi}_{x},\gamma) \, \mu(dx) 
> \int \MCov(\pi_{x}^{\sbm},\gamma) \, \mu(dx), 
\]
which is a contradiction to the optimality of $\pi^{\sbm}$.
\end{proof}


\section{The decomposition of the dual problem} \label{sec_gen_cas_tdotdp}

Let $(\psi_{n})_{n \geqslant 1} \subseteq \Cqaff$ be an optimizing sequence of convex functions. In the sequel, we have to assume that the values $(\mathcal{D}(\psi_{n}))_{n \geqslant 1}$ converge sufficiently fast to the optimal value $D(\mu,\nu)$. To this end, we define the \textit{suboptimality functional}\begin{equation} \label{def_func_subop}
\mathcal{S}(\psi) \coloneqq \mathcal{D}(\psi) - D(\mu,\nu), 
\qquad \psi \in \Cqaff,
\end{equation}
measuring the difference of $\mathcal{D}(\psi)$ from the optimal value $D(\mu,\nu)$. In Lemma \ref{prop_loc_fast_conv} below we will see that the ``fast convergence condition'' 
\begin{equation} \label{eq_def_psi_i}
\sum_{n \geqslant 1} \mathcal{S}(\psi_{n}) < +\infty
\end{equation}
on $(\psi_{n})_{n \geqslant 1}$, which we have already introduced in \eqref{eq_gen_case_fcc_fti} above, implies that the sequence $(\psi_{n})_{n \geqslant 1}$ is not only optimal for the pair $(\mu,\nu)$, but also for $\kappa^{\textnormal{B}}$-a.e.\ pair $(\mu^{I^{\textnormal{B}}},\nu^{I^{\textnormal{B}}})$.

\smallskip

For a fixed set $I^{\textnormal{B}} \in \widehat{\mathcal{K}}$ we consider, by analogy with \eqref{dual_prob_gen_sec_form}, the \textit{local} dual problem from $\mu^{I^{\textnormal{B}}}$ to $\nu^{I^{\textnormal{B}}}$
\[
D(\mu^{I^{\textnormal{B}}},\nu^{I^{\textnormal{B}}}) = 
\inf_{\substack{\psi \in \Cqaff, \\ \textnormal{$\psi$ convex}}}\mathcal{D}^{I^{\textnormal{B}}}(\psi),
\]
where the \textit{local} dual function $\mathcal{D}^{I^{\textnormal{B}}}(\, \cdot \,)$ now is given by
\[
\mathcal{D}^{I^{\textnormal{B}}}(\psi) \coloneqq \int \psi \, d\nu^{I^{\textnormal{B}}} - \int (\psi^{\ast} \ast \gamma)^{\ast} \, d \mu^{I^{\textnormal{B}}},
\qquad \psi \in \Cqaff.
\]
Similarly as in \eqref{def_func_subop}, we define the \textit{local} suboptimality functional
\[
\mathcal{S}^{I^{\textnormal{B}}}(\psi) \coloneqq \mathcal{D}^{I^{\textnormal{B}}}(\psi) - D(\mu^{I^{\textnormal{B}}},\nu^{I^{\textnormal{B}}}),
\qquad \psi \in \Cqaff.
\]

\begin{lemma}[\textsc{Decomposition of the dual problem}] \label{prop_loc_fast_conv} \
\begin{enumerate}[label=(\roman*)] 
\item \label{prop_loc_fast_conv_i} For a convex function $\psi \in \Cqaff$ we have 
\begin{equation} \label{eq_prop_01_g_c} 
\mathcal{S}(\psi) 
= \int_{\widehat{\mathcal{K}}} \, \mathcal{S}^{I^{\textnormal{B}}}(\psi) \ \kappa^{\textnormal{B}}(dI^{\textnormal{B}}).
\end{equation}
\item \label{prop_loc_fast_conv_ii} Let $(\psi_{n})_{n \geqslant 1} \subseteq \Cqaff$ be a sequence of convex functions satisfying \eqref{eq_def_psi_i}. Then, for $\kappa^{\textnormal{B}}$-a.e.\ $I^{\textnormal{B}} \in \widehat{\mathcal{K}}$, we have that $(\psi_{n})_{n \geqslant 1}$ is an optimizing sequence for the local dual problem from $\mu^{I^{\textnormal{B}}}$ to $\nu^{I^{\textnormal{B}}}$, i.e.
\begin{equation} \label{eq_prop_01_g_cd} 
D(\mu^{I^{\textnormal{B}}},\nu^{I^{\textnormal{B}}}) = \lim_{n \rightarrow \infty} \mathcal{D}^{I^{\textnormal{B}}}(\psi_{n}).
\end{equation}
\end{enumerate}
\begin{proof} \ref{prop_loc_fast_conv_i} We fix a convex function $\psi \in \Cqaff$ and want to show \eqref{eq_prop_01_g_c}. From Lemma \ref{lem_dotpsbmp} we know that
\[
P(\mu,\nu) 
= \int_{\widehat{\mathcal{K}}} \, P(\mu^{I^{\textnormal{B}}},\nu^{I^{\textnormal{B}}}) \ \kappa^{\textnormal{B}}(dI^{\textnormal{B}}).
\]
Therefore, since by Theorem \ref{theorem.1.4.o1a} we have $P(\mu,\nu) = D(\mu,\nu)$ and $P(\mu^{I^{\textnormal{B}}},\nu^{I^{\textnormal{B}}}) = D(\mu^{I^{\textnormal{B}}},\nu^{I^{\textnormal{B}}})$, for $\kappa^{\textnormal{B}}$-a.e.\ $I^{\textnormal{B}} \in \widehat{\mathcal{K}}$, it suffices to show that
\[
\mathcal{D}(\psi) 
= \int_{\widehat{\mathcal{K}}} \, \mathcal{D}^{I^{\textnormal{B}}}(\psi) \ \kappa^{\textnormal{B}}(dI^{\textnormal{B}}),
\]
which is immediate from \eqref{lem_dis_theo_eq_mu} and \eqref{eq_dis_nu_sbm}.

\smallskip

\noindent \ref{prop_loc_fast_conv_ii} Note that for any $\psi \in \Cqaff$ and for $\kappa^{\textnormal{B}}$-a.e.\ $I^{\textnormal{B}} \in \widehat{\mathcal{K}}$, we have $\mathcal{S}^{I^{\textnormal{B}}}(\psi) \geqslant 0$. By \eqref{eq_prop_01_g_c} and \eqref{eq_def_psi_i} we thus have
\begin{align*}
\sum_{n \geqslant 1} \mathcal{S}(\psi_{n}) 
&= \sum_{n \geqslant 1} \int_{\widehat{\mathcal{K}}} \, \mathcal{S}^{I^{\textnormal{B}}}(\psi_{n}) \ \kappa^{\textnormal{B}}(dI^{\textnormal{B}}) \\
&= \int_{\widehat{\mathcal{K}}} \, \Big( \sum_{n \geqslant 1} \mathcal{S}^{I^{\textnormal{B}}}(\psi_{n}) \Big) \ \kappa^{\textnormal{B}}(dI^{\textnormal{B}}) < +\infty.
\end{align*}
Therefore $\lim_{n \rightarrow \infty} \mathcal{S}^{I^{\textnormal{B}}}(\psi_{n}) = 0$, for $\kappa^{\textnormal{B}}$-a.e.\ $I^{\textnormal{B}} \in \widehat{\mathcal{K}}$, which gives \eqref{eq_prop_01_g_cd}.
\end{proof}
\end{lemma}
 
With the help of Lemma \ref{prop_loc_fast_conv}, we obtain the following consequence from Proposition \ref{prop_psinbar_bound_sharper}.

\begin{proposition} \label{lem_cruc_incl} Let $(\psi_{n})_{n \geqslant 1} \subseteq \Cqaff$ be a sequence of convex functions satisfying \eqref{eq_def_psi_i}. Then the following assertions hold for $\kappa^{\textnormal{B}}$-a.e.\ $I^{\textnormal{B}} \in \widehat{\mathcal{K}}$:
\begin{enumerate}[label=(\roman*)] 
\item \label{lem_cruc_incl_i_i} $\mu^{I^{\textnormal{B}}} \lc \nu^{I^{\textnormal{B}}}$, $\mu^{I^{\textnormal{B}}}(I^{\textnormal{B}}) = 1$, and 
\begin{equation} \label{lem_cruc_incl_i}
\widehat{\supp}(\nu^{I^{\textnormal{B}}}) = C^{\textnormal{B}},    
\end{equation}
where $\mu^{I^{\textnormal{B}}}$ and $\nu^{I^{\textnormal{B}}}$ are defined as in Lemma \ref{lem_dis_theo} and \eqref{eq_dis_nu_sbm_single}, respectively.
\item \label{lem_cruc_incl_i_i_i} There exist representatives $\bar{\psi}_{n}^{I^{\textnormal{B}}}\in [\psi_{n}]$ and there is a lower semicontinuous convex function $\psilim^{I^{\textnormal{B}}} \colon \Rd \rightarrow [0,+\infty]$ such that
\begin{alignat}{2}
\lim_{n \rightarrow \infty} \bar{\psi}_{n}^{I^{\textnormal{B}}}(y) &= \psilim^{I^{\textnormal{B}}}(y) < + \infty, \qquad &&y \in I^{\textnormal{B}}, \label{prop_psinbar_bound_sharperer_ii} \\
\lim_{n \rightarrow \infty} \bar{\psi}_{n}^{I^{\textnormal{B}}}(y) &=\psilim^{I^{\textnormal{B}}}(y) =  + \infty, \qquad &&y \in \Rd \setminus C^{\textnormal{B}}. \label{prop_psinbar_bound_sharperer_iii}
\end{alignat}
\item \label{lem_cruc_incl_i_i_i_i} The function $\psilim^{I^{\textnormal{B}}}$ is a dual optimizer for the pair $(\mu^{I^{\textnormal{B}}},\nu^{I^{\textnormal{B}}})$.
\end{enumerate}
\begin{proof} The following statements are valid for $\kappa^{\textnormal{B}}$-a.e.\ $I^{\textnormal{B}}$:

\smallskip

By Lemma \ref{lem_dis_theo}, we have $\mu^{I^{\textnormal{B}}}(I^{\textnormal{B}}) = 1$. From the definitions of $\mu^{I^{\textnormal{B}}}$ and $\nu^{I^{\textnormal{B}}}$ it follows that $\mu^{I^{\textnormal{B}}} \lc \nu^{I^{\textnormal{B}}}$. Moreover, by Lemma \ref{prop_loc_fast_conv}, $(\psi_{n})_{n \geqslant 1}$ is an optimizing sequence for the local dual problem from $\mu^{I^{\textnormal{B}}}$ to $\nu^{I^{\textnormal{B}}}$. Hence we can apply Proposition \ref{prop_psinbar_bound_sharper} to the pair $(\mu^{I^{\textnormal{B}}},\nu^{I^{\textnormal{B}}})$, which gives \ref{lem_cruc_incl_i_i} -- \ref{lem_cruc_incl_i_i_i_i}.
\end{proof}
\end{proposition}

\section{The proofs of Theorems \ref{irr_theom_gen_case_rep}, \ref{gen_case_theom_gen_case_int}, and Corollary \ref{gen_case_cor_gen_case_int_dmt}} \label{sec_1b_potmr}

Combining the results of Sections \ref{ecocford_sec_three} -- \ref{sec_gen_cas_tdotdp} allows us to finally provide the proof of our main result, Theorem \ref{gen_case_theom_gen_case_int}.

\begin{proof}[Proof of Theorem \ref{gen_case_theom_gen_case_int}:] As always, let $\mu, \nu \in \PP_{2}(\Rd)$ with $\mu \lc \nu$. Throughout the proof we fix a sequence $(\psi_{n})_{n \geqslant 1} \subseteq \Cqaff$ of convex functions satisfying \eqref{eq_def_psi_i} and consider the Bass paving $\{I^{\textnormal{B}}(x)\}_{x \in \Rd}$ of Definition \ref{def_sbm_paving}, induced by this sequence. In the following, all statements about a set $I^{\textnormal{B}} \in \widehat{\mathcal{K}}$ have to be understood in a $\kappa^{\textnormal{B}}$-a.s.\ sense. Once again we emphasize that the ``fast convergence condition'' \eqref{eq_def_psi_i} ensures that $(\psi_{n})_{n \geqslant 1}$ is an optimizing sequence for the local dual problems from $\mu^{I^{\textnormal{B}}}$ to $\nu^{I^{\textnormal{B}}}$. Note that the assertions \ref{gen_theom_gen_case_rep_o} -- \ref{gen_theom_gen_case_rep_ii} of Theorem \ref{gen_case_theom_gen_case_int} precisely correspond to the assertions \ref{lem_cruc_incl_i_i} -- \ref{lem_cruc_incl_i_i_i_i} of Proposition \ref{lem_cruc_incl}.

\smallskip

\noindent \ref{gen_theom_gen_case_rep_iii} As $\psilim^{I^{\textnormal{B}}}$ is a dual optimizer for the pair $(\mu^{I^{\textnormal{B}}},\nu^{I^{\textnormal{B}}})$, by Theorem \ref{theorem.1.4.o1a}, there is a Bass martingale $M^{I^{\textnormal{B}}} = (M_{t}^{I^{\textnormal{B}}})_{0 \leqslant t \leqslant 1}$ from $\mu^{I^{\textnormal{B}}}$ to $\nu^{I^{\textnormal{B}}}$. Moreover, the pair $(v^{I^{\textnormal{B}}},\alpha^{I^{\textnormal{B}}})$ defining the Bass martingale $M^{I^{\textnormal{B}}}$ is given as claimed in \eqref{eq_gen_theom_gen_case_rep_iii}. The Bass martingale is the stretched Brownian motion $\pi^{I^{\textnormal{B}}}$ from $\mu^{I^{\textnormal{B}}}$ to $\nu^{I^{\textnormal{B}}}$ and 
\[
\Law(M_{0}^{I^{\textnormal{B}}},M_{1}^{I^{\textnormal{B}}}) = \pi^{I^{\textnormal{B}}} \in \MT(\mu^{I^{\textnormal{B}}},\nu^{I^{\textnormal{B}}}), 
\]
where the martingale transport $\pi^{I^{\textnormal{B}}}$ is defined as in \eqref{eq_def_pisbm_loc}.

\smallskip

\noindent \ref{gen_theom_gen_case_rep_iv} The existence of the Bass martingale $M^{I^{\textnormal{B}}}$ from $\mu^{I^{\textnormal{B}}}$ to $\nu^{I^{\textnormal{B}}}$ implies the irreducibility of the pair $(\mu^{I^{\textnormal{B}}},\nu^{I^{\textnormal{B}}})$. 
\end{proof}

Combining the results of \cite{BVBST23} with Theorem \ref{gen_case_theom_gen_case_int}, we obtain Theorem \ref{irr_theom_gen_case_rep}.

\begin{proof}[Proof of Theorem \ref{irr_theom_gen_case_rep}] \ref{irr_theom_gen_case_rep_1} Since the pair $(\mu,\nu)$ is irreducible, by Theorem \ref{theorem.1.4.o1a} there is a dual optimizer $\psiopt$ with $\mu(\operatorname{ri}(\dom\psiopt)) = 1$, inducing a Bass martingale $(M_{t})_{0 \leqslant t \leqslant 1}$ from $\mu$ to $\nu$. More precisely, the pair $(v,\alpha)$ defining the Bass martingale is given by $v = \psiopt^{\ast}$ and $\alpha = \nabla (\psiopt^{\ast} \ast \gamma)^{\ast}(\mu)$. In particular,
\[
(\nabla \psiopt^{\ast} \ast \gamma)(\alpha) = \mu
\qquad \textnormal{ and } \qquad 
\nabla \psiopt^{\ast}(\alpha \ast \gamma) = \nu.
\]
As $I = \operatorname{ri}(\widehat{\supp}(\nu))$, this implies that $\operatorname{ri}(\dom\psiopt) = I$ and thus $\mu(I) = 1$.

\smallskip

Applying Theorem \ref{gen_case_theom_gen_case_int}, we conclude that $C = \widehat{\supp}(\nu) = C^{\textnormal{B}}(x)$, for $\mu$-a.e.\ $x \in \Rd$, and obtain the assertions \ref{irr_theom_gen_case_rep_o} -- \ref{irr_theom_gen_case_rep_iii} of Theorem \ref{gen_case_theom_gen_case_int}. Moreover, we see that $\psiopt$ is equal to $\psilim$, modulo adding an affine function.

\medskip

\noindent \ref{irr_theom_gen_case_rep_2} Without loss of generality we assume that the optimizing sequence $(\psi_{n})_{n \geqslant 1} \subseteq \Cqaff$ consists of convex functions $\psi_{n}$ which are nonnegative. By assumption, the relative interior $I$ of $C = \widehat{\supp}(\nu)$ has full $\mu$-measure. Furthermore, the sequence $(\psi_{n}(y))_{n \geqslant 1}$ is assumed to be bounded, for all $y \in I$. Then, by \cite[Proposition 7.20]{BVBST23} and \cite[Corollary 7.21]{BVBST23}, the pair $(\mu,\nu)$ is irreducible, so that again the assertions \ref{irr_theom_gen_case_rep_i} -- \ref{irr_theom_gen_case_rep_iii} hold. 
\end{proof}

\begin{proof}[Proof of Corollary \ref{gen_case_cor_gen_case_int_dmt}:] As in the proof of Theorem \ref{gen_case_theom_gen_case_int}, we fix a sequence $(\psi_{n})_{n \geqslant 1} \subseteq \Cqaff$ of convex functions satisfying \eqref{eq_def_psi_i} and consider the Bass paving $\{I^{\textnormal{B}}(x)\}_{x \in \Rd}$ of Definition \ref{def_sbm_paving}, induced by this sequence. By Theorem \ref{gen_case_theom_gen_case_int}, the pair $(\mu^{I^{\textnormal{B}}(x)},\nu^{I^{\textnormal{B}}(x)})$ is irreducible, for $\mu$-a.e.\ $x \in \Rd$. Thus, by \cite[Theorem D.1]{BVBST23}, the probability measure $\pi^{\sbm}_{x}$ is equivalent to $\nu^{I^{\textnormal{B}}(x)}$, for $\mu$-a.e.\ $x \in \Rd$. In particular, we have 
\[
\supp(\pi^{\sbm}_{x}) = \supp(\nu^{I^{\textnormal{B}}(x)}),
\]
for $\mu$-a.e.\ $x \in \Rd$. On the other hand, by \ref{gen_theom_gen_case_rep_o} of Theorem \ref{gen_case_theom_gen_case_int}, we have that 
\[
C^{\textnormal{B}}(x) = \widehat{\supp}(\nu^{I^{\textnormal{B}}(x)}),
\]
for $\mu$-a.e.\ $x \in \Rd$. We conclude that $\widehat{\supp}(\pi_{x}^{\sbm}) = C^{\textnormal{B}}(x)$, for $\mu$-a.e.\ $x \in \Rd$, i.e.\ the equality in \eqref{gen_case_cor_gen_case_int_dmt_eq}. 

\smallskip

In order to show the inclusion in \eqref{gen_case_cor_gen_case_int_dmt_eq}, we let $\pi \in \MT(\mu,\nu)$ be an arbitrary martingale transport. By contradiction, we assume that the set
\[
A \coloneqq \Big\{ x \in \Rd \colon \widehat{\supp}(\pi_{x}) \cap \big ( \Rd \setminus C^{\textnormal{B}}(x) \big)  \neq \varnothing \Big\}
\]
has strictly positive $\mu$-measure. By \ref{irr_theom_gen_case_rep_i} of Theorem \ref{irr_theom_gen_case_rep}, for $\mu$-a.e.\ $x \in A$, we can select representatives $\bar{\psi}_{n}^{x} \in [\psi_{n}]$ and there is a lower semicontinuous convex function $\psilim^{x} \colon \Rd \rightarrow [0,+\infty]$, such that 
\begin{alignat}{2}
\lim_{n \rightarrow \infty} \bar{\psi}_{n}^{x}(y) &= \psilim^{x}(y) < + \infty, \qquad &&y \in I^{\textnormal{B}}(x), \label{gen_case_cor_gen_case_int_dmt_i} \\
\lim_{n \rightarrow \infty} \bar{\psi}_{n}^{x}(y) &= \psilim^{x}(y) = + \infty, \qquad &&y \in \Rd \setminus C^{\textnormal{B}}(x). \label{gen_case_cor_gen_case_int_dmt_ii}
\end{alignat}
This selection can be done in a $\mu$-measurable way. In particular, by \eqref{gen_case_cor_gen_case_int_dmt_i} and since $x \in I^{\textnormal{B}}(x)$, the sequence $(\bar{\psi}_{n}^{x}(x))_{n \geqslant 1}$ is bounded, for $\mu$-a.e.\ $x \in A$. Next, we choose $\tilde{A} \subseteq A$ with $\mu(\tilde{A}) > 0$ such that the sequence $(\bar{\psi}_{n}^{x}(x))_{n \geqslant 1}$ is uniformly bounded by some $M > 0$, for all $x \in \tilde{A}$. Since
\[
\int \big( \bar{\psi}_{n}^{x}(y) - \bar{\psi}_{n}^{x}(x) \big) \, \pi_{x}(dy) \geqslant 0,
\]
for $\mu$-a.e.\ $x \in \Rd$, we conclude that
\[
\int \int \big( \bar{\psi}_{n}^{x}(y) - \bar{\psi}_{n}^{x}(x) \big) \, \pi_{x}(dy) \, \mu(dx) 
\geqslant 
\int_{x \in \tilde{A}} \int_{y \in \Rd}  \bar{\psi}_{n}^{x}(y) \, \pi_{x}(dy) \, \mu(dx) - M.
\]
Similarly, as $\bar{\psi}_{n}^{x}(y) \geqslant 0$, we have that
\[
\int_{x \in \tilde{A}} \int_{y \in \Rd}  \bar{\psi}_{n}^{x}(y) \, \pi_{x}(dy) \, \mu(dx)
\geqslant 
\int_{x \in \tilde{A}} \int_{y \in \Rd \setminus C^{\textnormal{B}}(x)} \bar{\psi}_{n}^{x}(y) \, \pi_{x}(dy) \, \mu(dx).
\]
By Fatou's lemma and using \eqref{gen_case_cor_gen_case_int_dmt_ii}, we obtain that
\[
\liminf_{n \rightarrow \infty} \int_{x \in \tilde{A}} \int_{y \in \Rd \setminus C^{\textnormal{B}}(x)} \bar{\psi}_{n}^{x}(y) \, \pi_{x}(dy) \, \mu(dx) = + \infty.
\]
Altogether, we have
\[
\sup_{n \geqslant 1} \int \int \big( \bar{\psi}_{n}^{x}(y) - \bar{\psi}_{n}^{x}(x) \big) \, \pi_{x}(dy) \, \mu(dx) = + \infty,
\]
which is a contradiction to \cite[Lemma 7.9]{BVBST23}.    
\end{proof}

\section{Examples} \label{section_examples}

We start with a simple example motivating the treatment of the general case in Theorem \ref{gen_case_theom_gen_case_int} when the Bass paving $\{I^{\textnormal{B}}(x)\}_{x \in \R^d}$ of the pair $(\mu,\nu)$ is not reduced to one element so that there does not exist a Bass martingale from $\mu$ to $\nu$.

\begin{example} \label{ex_two_bm_gen_case} We first spell out the case of a geometric Brownian motion on $\R_{+} = (0,+\infty)$, which is an example of a Bass martingale (recall Definition \ref{def_bm_gcp}).

\smallskip

Let $Z$ be a standard Gaussian random variable on $\R$ and define 
\[
\mu_{+} \coloneqq \delta_1,  \qquad 
\nu_{+} \coloneqq \operatorname{Law}\big(\exp(Z - \tfrac{1}{2})\big),
\]
so that $\mu_{+} \lc \nu_{+}$. It is straightforward to calculate all the functions defining the Bass martingale from $\mu_{+}$ to $\nu_{+}$. In particular, we have 
\[
v_{+}(z) = \tfrac{d}{dz} v_{+}(z) = \exp(z-\tfrac{1}{2}), \qquad z \in \R,
\]
with the property that $\operatorname{Law}(v_{+}(Z)) = \nu_{+}$. The convex conjugate of $v_{+}$ and its derivative are equal to
\[
\psi_{+}(y) = y (\log y  - \tfrac{1}{2}), \qquad \tfrac{d}{dy} \psi_{+}(y) = \log y + \tfrac{1}{2}, \qquad y > 0.
\]
Summing up, $\pi_{+} = \delta_{1} \otimes \nu_{+}$ defines a Bass martingale from $\mu_{+}$ to $\nu_{+}$, for which we can explicitly compute all its ingredients.

\smallskip

Next, we mirror this example along the vertical axis, i.e.
\[
\mu_{-}  \coloneqq \delta_{-1},  \qquad 
\nu_{-} \coloneqq \operatorname{Law}\big(-\exp(Z - \tfrac{1}{2})\big).
\]
Again, we can explicitly calculate the relevant quantities of the Bass martingale joining $\mu_{-}$ to $\nu_{-}$ and obtain the mirrored functions on $\R_{-} = (-\infty,0)$. In particular,
\[
\psi_{-}(y) = - y \big(\log (-y)  - \tfrac{1}{2}\big), \qquad y < 0.
\]

\smallskip

Finally, we define the convex combination of these two examples, i.e.
\[
\mu \coloneqq \tfrac{1}{2}(\mu_{+} + \mu_{-}), \qquad 
\nu \coloneqq \tfrac{1}{2}(\nu_{+} + \nu_{-}).
\]
It is still possible to pin down the primal optimizer $\pi^{\sbm}$ for the pair $(\mu,\nu)$, namely
\[
\pi^{\sbm} = \tfrac{1}{2} ( \delta_{1} \otimes \nu_{-}+ \delta_{-1} \otimes \nu_{+}).
\]
However, we do not have a Bass martingale any more but only a stretched Brownian motion from $\mu$ to $\nu$, since we cannot find a function inducing the transport $\pi^{\sbm}$. Intuitively speaking, the stretched Brownian motion consists of the above two Bass martingales which ``do not talk to each other''. The Bass paving is given by
\[
\big\{I^{\textnormal{B}}(x)\big\}_{x \in \R^d} = \big\{ (-\infty,0) , \{0\}, (0,+\infty) \big\},
\]
with $\kappa^{\textnormal{B}}((-\infty,0)) = \kappa^{\textnormal{B}}((0,+\infty)) = \frac{1}{2}$ and $\kappa^{\textnormal{B}}(\{0\}) = 0$. We remark that in the one-dimensional case such invariant decompositions were already studied in \cite{BJ16}.

\smallskip

Regarding the dual side of optimization, we know from \eqref{dual_prob_gen_sec_form} that there is an optimizing sequence $(\psi_n)_{n \geqslant 1} \subseteq \Cqaffo$ of convex functions for the dual problem \eqref{dp_mbb_mcov_dt_f}. On $\R_{+}$ the functions $(\psi_n)_{n \geqslant 1}$ should approximate the dual optimizer $\psi_{+}$ for $(\mu_{+},\nu_{+})$, while on $\R_{-}$ they should approximate the dual optimizer $\psi_{-}$ for $(\mu_{-},\nu_{-})$. A naive guess would be to choose (or approximate) the function
\begin{equation} \label{eq_gen_ex_c}
\psi(y) \coloneqq 
\begin{cases} 
\psi_{+}(y), & \mbox{for } y \geqslant 0, \\
\psi_{-}(y), & \mbox{for } y \leqslant 0.
\end{cases} 
\end{equation}
Two problems arise when one tries to approximate $\psi$ as defined in \eqref{eq_gen_ex_c} by convex functions $\psi_{n}$ in $\Cqaffo$: firstly at $0$ (convexity) and secondly at $\pm \infty$ (boundedness). The boundedness requirement for $\psi_{n} \in \Cqaffo$ is easily taken care of and we ignore this issue. The crucial point is that, by trying to paste $\psi_{+}$ and $\psi_{-}$ together at $0$, we grossly violate the convexity property, as $\psi^{\prime}(0_{+}) = - \infty$ while $\psi^{\prime}(0_{-}) = + \infty$.

\smallskip

Here is the remedy to this problem: let $\psi_{n,+}$ and $\psi_{n,-}$ be convex approximations of $\psi_{+}$ and $\psi_{-}$, respectively, which both have a finite derivative at $y = 0$ (which clearly is possible). Defining
\[
\tilde{\psi}_{n}(y) \coloneqq 
\begin{cases} 
\psi_{n,+}(y), & \mbox{for } y \geqslant 0, \\
\psi_{n,-}(y), & \mbox{for } y \leqslant 0,
\end{cases} 
\]
we still have not found a sequence of convex functions approximating $\psi_{+}$ and $\psi_{-}$ on $\R_{+}$ and $\R_{-}$, respectively. But for large enough $C_{n}$, the function
\[
\psi_{n}(y) \coloneqq \tilde{\psi}_{n}(y) + C_{n} \vert y \vert 
\]
becomes convex. While $\psi_{n}$, of course, neither approximates $\psi_{+}$ on $\R_{+}$ nor $\psi_{-}$ on $\R_{-}$, recall that the dual functions should be considered ``modulo adding affine functions''. Passing to
\[
\psi_{n}(y) - C_{n} y, \qquad y \geqslant 0
\]
we get a good approximation of $\psi_{+}$ on $\R_{+}$, and passing to
\[
\psi_{n}(y) + C_{n} y, \qquad y \leqslant 0
\]
we get a good approximation of $\psi_{-}$ on $\R_{-}$. Hence ``locally'', i.e.\ on $\R_{+}$ and $\R_{-}$, the sequence $(\psi_{n})_{n \geqslant 1}$ converges --- after adding proper affine functions --- to $\psi_{+}$ and $\psi_{-}$, respectively. But, of course, the sequence $(\psi_{n})_{n \geqslant 1}$ itself does not converge on $\R$ and it should come as no surprise that there is no ``global'' optimizer $\psi$ defined on $\R$. \hfill $\Diamond$ 
\end{example}

\begin{example} \label{example_circles} We thank Mathias Beiglb\"ock and Krzysztof Ciosmak for kindly suggesting this example to us, the former for the case $R > \frac{3}{2}$, the latter for the case $R = \frac{3}{2}$.

\medskip

We denote by $\boldsymbol{S}_{r}$ the sphere in $\mathds{R}^{2}$ with radius $r$, i.e.
\[
\boldsymbol{S}_{r} \coloneqq  \{ x \in \mathds{R}^{2} \colon \vert x \vert = r \}.
\]
We let $\mu$ be the uniform distribution on $\boldsymbol{S}_{1}$ and define the probability measure 
\[ 
\nu^{(R)} \coloneqq \tfrac{1}{2} ( \nu_{\frac{1}{2}} + \nu_{R} ),
\]
where $\nu_{\frac{1}{2}}$ and $\nu_{R}$ are uniform distributions on $\boldsymbol{S}_{\frac{1}{2}}$ and $\boldsymbol{S}_{R}$, respectively. It is straightforward to check --- and will follow from the discussion below --- that $\nu^{(R)}$ dominates $\mu$ in convex order if and only if $R \geqslant \frac{3}{2}$.

\medskip

\noindent \fbox{$R = \frac{3}{2}$} Let us first examine the case $R = \frac{3}{2}$. For $x \in \boldsymbol{S}_{1}$, we define the probability measure 
\[
\pi_{x}^{(\frac{3}{2})} \coloneqq \tfrac{1}{2} ( \delta_{\frac{x}{2}} + \delta_{\frac{3x}{2}} )
\]
and claim that the unique martingale transport from $\mu$ to $\nu^{(\frac{3}{2})}$ is given by
\[
\pi^{(\frac{3}{2})}(dx,dy) \coloneqq \pi_{x}^{(\frac{3}{2})}(dy)  \, \mu(dx) \in \MT(\mu,\nu^{(\frac{3}{2})}).
\]
Clearly, $\pi^{(\frac{3}{2})}$ is a martingale transport between $\mu$ and $\nu^{(\frac{3}{2})}$. To see that $\pi^{(\frac{3}{2})}$ is in fact the unique element of $\MT(\mu,\nu^{(\frac{3}{2})})$, we consider the convex function $\psi(x) = \vert x \vert$ and note that
\[
\int \psi \, d\mu = \int \psi \, d\nu^{(\frac{3}{2})} = 1.
\]
For any candidate martingale transport $\pi(dx,dy) = \pi_{x}(dy)  \, \mu(dx) \in \MT(\mu,\nu^{(\frac{3}{2})})$ we thus obtain from Jensen's inequality that
\[
\int \psi \, d\pi_{x}(dy) = \psi(x),
\]
which implies that $\pi_{x} = \pi_{x}^{(\frac{3}{2})}$, for $\mu$-a.e.\ $x \in \mathds{R}^{2}$.

\smallskip

In particular, the pair $(\mu,\nu^{(\frac{3}{2})})$ is not irreducible. In fact, identifying $\mathds{R}^{2}$ with $\mathds{C}$, the Bass paving consists ($\kappa^{\textnormal{B}}$-almost surely) of the continuum many relatively open line segments 
\[
I_{\varphi} 
\coloneqq I^{\textnormal{B}}(\mathrm{e}^{i \varphi}) 
= \big\{ r \mathrm{e}^{i \varphi} \colon  r \in (\tfrac{1}{2}, \tfrac{3}{2}) \big\}, \qquad \varphi \in [0,2\pi),
\]
as $\pi^{(\frac{3}{2})}$ transports $I_{\varphi}$ into $C_{\varphi} \coloneqq \overline{I_{\varphi}}$.

\medskip 

\noindent \fbox{$R > \frac{3}{2}$} We next pass to the case $R > \frac{3}{2}$, which behaves quite differently. For $\alpha \in [-\frac{\pi}{6},\frac{\pi}{6}]$, the line 
\begin{equation} \label{eq_line_elloft}
t \longmapsto \ell(t) \coloneqq (1,0) + t(1,\vartheta),
\end{equation}
where $\vartheta = \tan \alpha$, intersects the sphere $\boldsymbol{S}_{\frac{1}{2}}$ at a point $x^{(\frac{1}{2})}$, as indicated in Figure \ref{figure_a} below. For $\alpha \in (-\frac{\pi}{6},\frac{\pi}{6})$, there are two intersection points and we choose $x^{(\frac{1}{2})}$ as the one which is closer to $(1,0)$. Furthermore, we denote by $x^{(R)}$ the intersection point of the line $t \mapsto \ell(t)$ with the sphere $\boldsymbol{S}_{R}$, where we choose $x^{(R)}$ as the intersection point to the right of $(1,0)$. Elementary geometry reveals that, for $\alpha \in (0,\frac{\pi}{6}]$, there is a unique $R = R(\alpha) > \frac{3}{2}$ such that 
\[
\vert (1,0) - x^{(\frac{1}{2})} \vert 
= \vert (1,0) - x^{(R)} \vert.
\]
In other words, $R$ is chosen such that $(1,0)$ is the midpoint between $x^{(\frac{1}{2})}$ and $x^{(R)}$.

\begin{figure}[H] 
\centering
\includegraphics[width=0.74\textwidth]{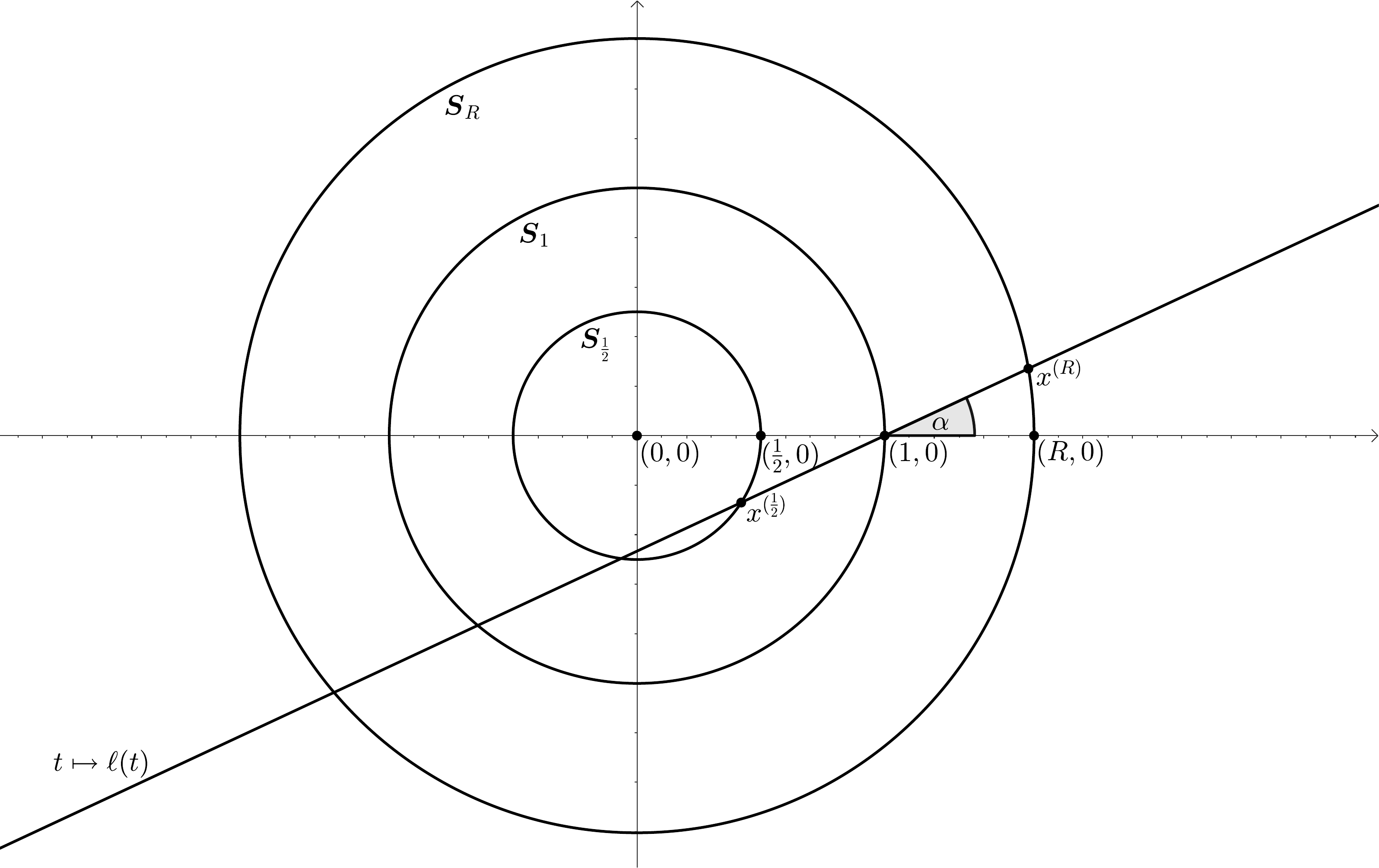}
\caption{Illustration of Example \ref{example_circles} in the case $\alpha = 25^{\circ}$ and $R \approx \frac{8}{5}$.}
\label{figure_a}
\end{figure}

\smallskip

After these preparations, we turn to the martingale transports in $\MT(\mu,\nu^{(R)})$, for $R = R(\alpha) > \frac{3}{2}$ as above. Figure \ref{figure_a} indicates a martingale transport analogous to the one constructed above in the case $R = \frac{3}{2}$. Identifying again $\mathds{R}^{2}$ with $\mathds{C}$, we define
\[
\pi_{\mathrm{e}^{i\varphi}}^{(R)} \coloneqq \tfrac{1}{2} ( \delta_{\mathrm{e}^{i\varphi}x^{(\frac{1}{2})}} + \delta_{\mathrm{e}^{i\varphi}x^{(R)}}),
\qquad \varphi \in [0,2\pi).
\]
As in the case $R = \frac{3}{2}$ it follows from rotational invariance that
\[
\pi^{(R)}(dx,dy) \coloneqq \pi_{x}^{(R)}(dy)  \, \mu(dx) \in \MT(\mu,\nu^{(R)})
\]
defines a martingale transport from $\mu$ to $\nu^{(R)}$. What about the invariant cells attached to the martingale transport $\pi^{(R)}$? Again as in the case $R = \frac{3}{2}$, there are continuum many relatively open line segments 
\begin{equation} \label{ex_sec_ca_inc_ce}
I_{\varphi} 
\coloneqq I^{\textnormal{B}}(\mathrm{e}^{i \varphi}) 
= \big( \mathrm{e}^{i \varphi} x^{(\frac{1}{2})}, \mathrm{e}^{i \varphi} x^{(R)} \big), \qquad \varphi \in [0,2\pi),
\end{equation}
such that $\pi^{(R)}$ leaves $C_{\varphi} \coloneqq \overline{I_{\varphi}}$ invariant.

\smallskip

But now the transport $\pi^{(R)}$ is not unique any more. A striking way to see this is to replace $\alpha > 0$ by $-\alpha$, i.e.\ mirroring the line $t \mapsto \ell(t)$ defined in \eqref{eq_line_elloft} along the $x$-axis. In this way we obtain another martingale transport, denoted by $\overline{\pi}^{(R)}$, which is different from $\pi^{(R)}$. In particular, the invariant cells of $\overline{\pi}^{(R)}$ are obtained by ``flipping'' the cells in \eqref{ex_sec_ca_inc_ce}. This phenomenon is in sharp contrast to the one-dimensional setting $d = 1$, analyzed in \cite{BJ16}, where it was shown that the paving of $\mathds{R}$ into invariant cells only depends on the pair $(\mu,\nu)$, but not on the special choice of $\pi \in \MT(\mu,\nu)$. In the case $d \geqslant 2$, we refer to \cite{GKL19} for an analysis of the invariant pavings of $\mathds{R}^{d}$ which do depend on the choice of $\pi \in \MT(\mu,\nu)$, for $\mu, \nu \in \PP_{2}(\Rd)$ with $\mu \lc \nu$. In particular, our example in dimension $d=2$ shows that the invariant cells of a martingale transport $\pi \in \MT(\mu,\nu^{(R)})$ may depend on the choice of $\pi$.

\medskip

The above discussion motivates the approach of \cite{DMT19} and \cite{OS17}, which focused on the \textit{universal} partition of $\mathds{R}^{d}$, induced by $\mu, \nu \in \PP_{2}(\Rd)$ with $\mu \lc \nu$, as defined in \cite[Theorem 2.1]{DMT19} and revisited in Corollary \ref{gen_case_cor_gen_case_int_dmt} above. It turns out that in the above example of the pair $(\mu,\nu^{(R)})$, the De March--Touzi paving is in fact trivial, i.e.\ the pair $(\mu,\nu^{(R)})$ is irreducible. More precisely, we will see that there is a Bass martingale $M = (M_{t})_{0 \leqslant t \leqslant 1}$ from $\mu$ to $\nu^{(R)}$, for which we have that the law of $M_{t}$ is equivalent to Lebesgue measure on the closed disk with radius $R$ in $\mathds{R}^{2}$, for $0 < t < 1$, while, of course, $\Law(M_{0}) = \mu$ and $\Law(M_{1}) = \nu^{(R)}$. 

\medskip

We now consider the dual optimization problem for $(\mu,\nu^{(R)})$ in the case $R > \frac{3}{2}$, which can be explicitly solved. Let $\alpha^{(\rho)}$ denote the uniform distribution on the sphere $\boldsymbol{S}_{\rho}$ with radius $\rho > 0$ to be specified later. Denoting by $\gamma$ the standard Gaussian measure on $\mathds{R}^{2}$, we consider the convolution measure $\alpha^{(\rho)} \ast \gamma$ and define the Brenier map $\nabla v^{(\rho,R)}$ from $\alpha^{(\rho)} \ast \gamma$ to $\nu^{(R)}$. Here $v^{(\rho,R)}$ is a convex function on $\mathds{R}^{2}$, unique up to an additive constant, for which the derivative $\nabla v^{(\rho,R)}$ exists Lebesgue-almost everywhere and therefore also $(\alpha^{(\rho)} \ast \gamma)$-almost surely. As $\nu^{(R)}$ is supported by $\boldsymbol{S}_{\frac{1}{2}}$ and $\boldsymbol{S}_{R}$, we have that $\nabla v^{(\rho,R)}$ takes its values almost surely in $\boldsymbol{S}_{\frac{1}{2}} \cup \boldsymbol{S}_{R}$.

\smallskip

It is intuitively rather obvious that on each radial ray of $\mathds{R}^{2}$, say on $\{ \lambda(1,0) \colon \lambda > 0\}$, the directional derivative along this ray has to take either the value $\frac{1}{2}$ or $R$, each with conditional probability $\frac{1}{2}$ under the measure $\alpha^{(\rho)} \ast \gamma$, induced on this ray. Restricting $v^{(\rho,R)}$ to the $x$-axis, we obtain the following picture.

\begin{figure}[H] 
\centering
\includegraphics[width=0.74\textwidth]{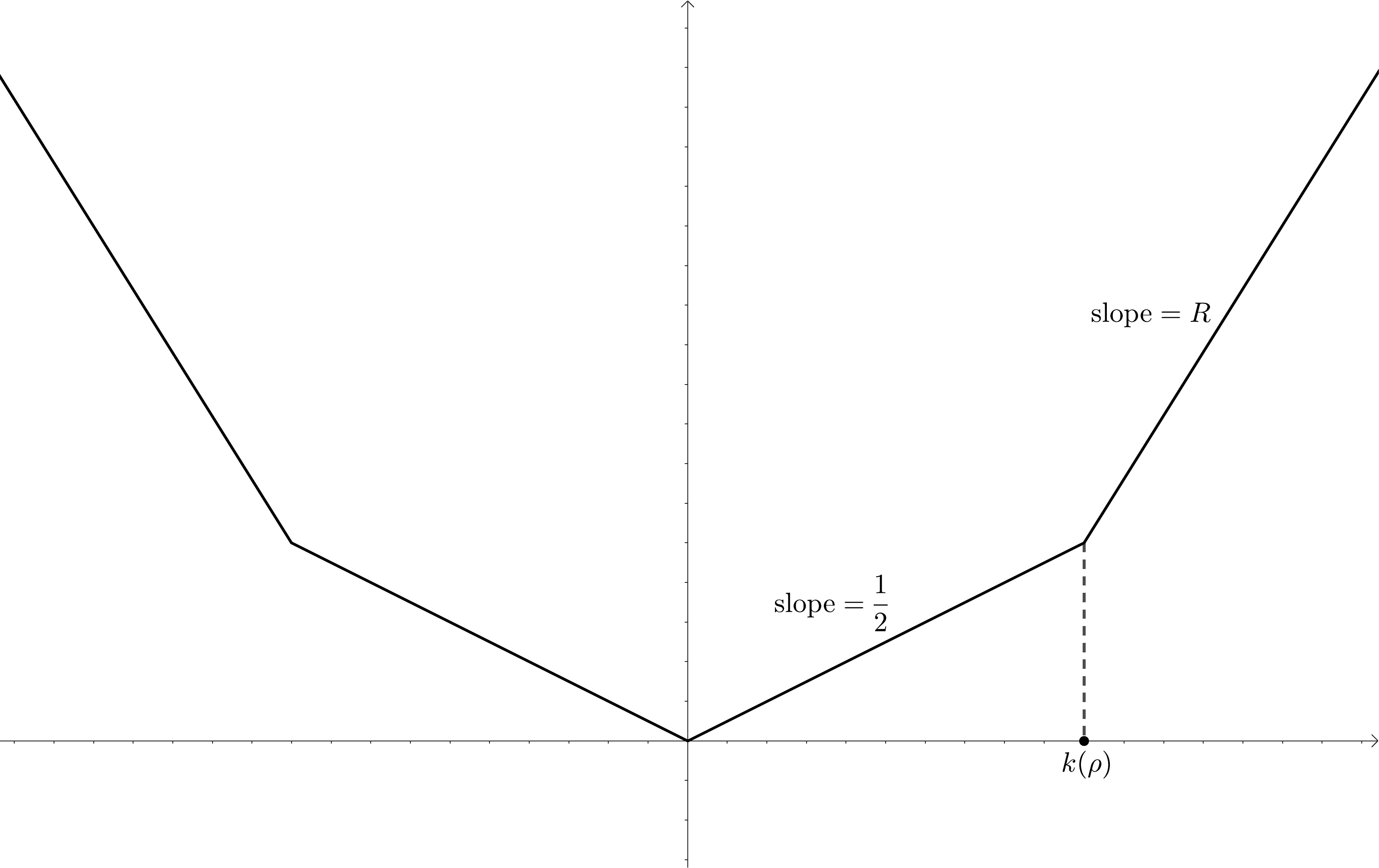}
\caption{Restriction of the convex function $v^{(\rho,R)}$ to the $x$-axis, with $R \approx \frac{8}{5}$.}
\label{figure_b}
\end{figure}

The function in Figure \ref{figure_b} has a kink at some point $k(\rho)$, which can be determined as the median of the conditional law of $\alpha^{(\rho)} \ast \gamma$, induced on the ray $\{ \lambda(1,0) \colon \lambda > 0\}$. Finally we determine, for $\rho > 0$, the number $R = R(\rho) > \frac{3}{2}$ as the unique number such that the rotationally invariant function $\nabla v^{(\rho,R)} \ast \gamma$ maps the measure $\delta_{\rho(0,1)} \ast \gamma$ on $\mathds{R}^{2}$ to the measure 
\[
(\nabla v^{(\rho,R)} \ast \gamma)(\delta_{\rho(0,1)} \ast \gamma)
\]
with barycenter $(0,1)$. We know from the general theory that there must be a unique such $R > \frac{3}{2}$ for any given $\rho > 0$. It also follows from the general theory that, for $0 < t < 1$, the law of $M_{t}$ is equivalent to Lebesgue measure on the closed disk with radius $R$ in $\mathds{R}^{2}$, i.e.\
\[
\{ x \in \mathds{R}^{2} \colon \vert x \vert \leqslant R\}.
\]
Indeed, by \cite[Remark 6.3, (iii)]{BVBST23} the law of $M_{t}$ is the image of $\alpha \ast \gamma^{t}$ under the map $\nabla v^{(\rho,R)} \ast \gamma^{1-t}$. While $\alpha \ast \gamma^{t}$ is equivalent to Lebesgue measure on $\mathds{R}^{2}$, for $0 < t \leqslant 1$, the range of $\nabla v^{(\rho,R)} \ast \gamma^{1-t}$ is the open ball of radius $R$ on $\mathds{R}^{2}$, for $0 \leqslant t < 1$. In conclusion, the support of the law of $M_{t}$ is the closed disk of radius $R$ in $\mathds{R}^{2}$, for $0 < t < 1$.

\smallskip

On a last note, we discuss the limiting behaviour of the function $\rho \mapsto R(\rho)$, for $\rho > 0$. It turns out that $\lim_{\rho \rightarrow +\infty} R(\rho) = \frac{3}{2}$ while $\lim_{\rho \rightarrow 0} R(\rho) = +\infty$. \hfill $\Diamond$ 
\end{example}

\newpage

\bibliographystyle{alpha}
{\footnotesize
\bibliography{references}}

\end{document}